\documentclass[a4paper,11pt,oneside,reqno]{amsart} %<<<
\usepackage[width=15cm,height=24cm]{geometry}
\usepackage{amsfonts,amssymb,amsmath,amsthm}
\usepackage{graphicx}
\usepackage{hyperref}

\newtheorem{theorem}{Theorem}[section]
\newtheorem{lemma}[theorem]{Lemma}

\newtheorem{corollary}[theorem]{Corollary}

\theoremstyle{definition}

\theoremstyle{remark}
\newtheorem{remark}[theorem]{Remark}

\numberwithin{equation}{section}
\def\Var{\mathop{\rm Var}\nolimits}
\def\Cov{\mathop{\rm Cov}\nolimits}

\def\sign{\mathop{\rm sign}\nolimits}
\def\Ai{\mathop{\rm Ai}\nolimits}
\def\d{\mathrm{d}}

\def\<{\langle}
\def\>{\rangle}

\newcommand{\N}{\ensuremath{\mathbb{N}}}
\newcommand{\R}{\ensuremath{\mathbb{R}}}

\newcommand{\E}{\ensuremath{\mathbb{E}}}
\newcommand{\Pro}{\ensuremath{\mathbb{P}}}
\newcommand{\indi}{\ensuremath{\boldsymbol 1}}

\newcommand{\GOE}{{\text{GOE}}}
\newcommand{\TAP}{{\text{TAP}}}
\newcommand{\simd}{\overset{\delta }\sim}
\newcommand{\Crt}{\mathop{\mathrm{Crt}}\nolimits}
\begin{document}

% Preamble  %<<<
\title{Random Matrices and complexity of Spin Glasses}
\author[A. Auffinger]{Antonio Auffinger}
\address{A. Auffinger\\
  Courant Institute of Mathematical Sciences\\
  New York University\\
  251 Mercer Street\\
  New York, NY 10012, USA}
\email{auffing@cims.nyu.edu}

\author[G. Ben Arous]{G\'erard Ben Arous}
\address{G. Ben Arous\\
  Courant Institute of Mathematical Sciences\\
  New York University\\
  251 Mercer Street\\
  New York, NY 10012, USA}
\email{benarous@cims.nyu.edu}

\author[J. \v Cern\'y]{Ji\v r\'\i~\v Cern\'y}
\address{J. \v Cern\'y\\
  Department of Mathematics\\
  ETH Zurich\\
  Raemistr. 101\\
  8092 Zurich\\
  Switzerland}
\email{jiri.cerny@math.ethz.ch}

%\subjclass[2000]{82C44,60K35,60G70,XXXXX}
%\keywords{XXXX}

\date{\today}
\begin{abstract}
  We give an asymptotic evaluation of the complexity of spherical $p$-spin
  spin-glass models via random matrix theory. This study enables us to
  obtain detailed information about the bottom of the energy landscape,
  including the absolute minimum (the ground state), the other  local
  minima, and describe an interesting layered structure of the low critical
  values for
  the Hamiltonians of these models. We also show that our approach allows
  us to compute the related TAP-complexity and extend the results known in the physics literature.
   As an independent tool, we
  prove a LDP for the $k$-th largest eigenvalue of the GOE, extending the
  results of \cite{BDG01}.
\end{abstract}
\maketitle
%>>>

\section{Introduction} %<<<
How many critical values does a typical random Morse function have on a high dimensional manifold? How many of given index, or below a given level? What is the topology of level sets? Our work
addresses the first two questions above for a class of natural random Gaussian functions on the  $N$-dimensional
sphere, known as $p$-spin spherical spin glass models in the physics literature and as isotropic models in the Gaussian process literature. The third question is covered in the forthcoming paper \cite{TG11}. 

We study here the complexity of these random functions, i.e the exponential behavior of the mean number of critical points, and more importantly of the mean number of critical points of given index in a given level set.
We introduce a new identity, based on the classical Kac-Rice formula,
relating random matrix theory and the problem of counting these critical values.
Using this identity and tools from random matrix theory (mainly large deviation results), we give an asymptotic evaluation  of the
complexity of these  spherical spin-glass models. Our study includes the important question of counting the mean number of
local minima below a given level, and in particular the question of
finding the ground state energy (the minimal value of the Hamiltonian). We show
that this question is directly related to the study of the edge of the
spectrum of the Gaussian Orthogonal Ensemble (GOE). 

The question of computing the complexity of mean-field spin glass models has 
recently been thoroughly studied in the physics literature (see for example 
  \cite{CLR03} and the references therein), mainly for a different measure of 
the complexity, i.e.~the mean number of solutions to the 
Thouless-Anderson-Palmer equations, or TAP-complexity. Our approach to the 
complexity enables us to recover known results in the physics literature about 
TAP-complexity, to compute the ground state energy (when p is even), and to describe an interesting 
layered structure of the low energy levels of the Hamiltonians of these models, 
which might prove useful for the study of the metastability of Langevin 
dynamics for these models (in longer time scales than those studied in 
  \cite{BDG01}).

The paper is organized as follows
In Section~\ref{s:notation}, we give our main results. In 
Section~\ref{s:exact}, we prove two main formulas (Theorem
  \ref{t:exactk} 
  and \ref{t:exactglobal}), relating random matrix theory (specifically the GOE) and 
spherical spin glasses. These formulas are consequences of the classical Kac-Rice 
formula (see  \cite{AT07} and \cite{AW09} for two excellent recent books giving a
very complete account of this formula and its consequences).The version we use here
is proved in section 12.1 of the book \cite{AT07}. For another modern account of the
Kac-Rice formula and similar techniques to those of section \ref{s:exact} see
Chapter 6 and section 8.3 of \cite{AW09}). 
The main ingredient to derive results from the Kac-Rice formula is the fact that,
for spherical spin-glass models, the Hessian of 
the Hamiltonian at a critical point, conditioned on the value of the 
Hamiltonian, is a symmetric Gaussian random matrix with independent entries (up to
symmetry) plus a diagonal matrix. This implies, in particular, that it is 
possible to relate statistics of critical points of index $k$ to statistics 
of the $k+1$-th smallest eigenvalue of a matrix sampled from the GOE.

In Section~\ref{s:ldp}, we compute precise logarithmic estimates of the 
complexity using the known large deviation principle (LDP) for the empirical 
spectral measure \cite{BG97} and for the largest eigenvalue of the GOE 
\cite{BDG01}. In fact we need a simple extension of the last LDP, i.e.~an LDP 
for the law of the $k$-th largest eigenvalue, which is of independent interest 
and proven in Appendix A.  

In Section~\ref{s:gs}, we show how these logarithmic results can be used to 
extract information about the lowest lying critical values. We first prove that 
the lowest lying critical points have an interesting  layered structure, 
Theorem \ref{t:belowEk}. We then show how our logarithmic results imply a lower 
bound on the ground state energy (the minimal value of the Hamiltonian). At 
this point it would be useful to have a concentration result for the number of 
local minima, for instance using a control of its second moment. Unfortunately, 
we cannot prove directly such a concentration result. Nevertheless we prove 
that our lower bound is tight (for p even), by proving the corresponding upper 
bound, using the Parisi formula for the free energy at positive temperature, as 
established by Talagrand  \cite{Tal06}.  It is remarkable that the ground state 
is indeed correctly predicted by our very naive approach, i.e.~by the vanishing 
of the ``annealed'' complexity of the number of local minima.  We expect this to 
be true for all models where Parisi's one-step replica symmetry breaking holds 
at low temperature. In Section~\ref{s:TAP}, we extend our results to the 
TAP-complexity and compare our results to the physics literature \cite{CLR03}, 
\cite{CS95}. 

In Section~\ref{s:sharp}, we show how one can go further and obtain sharper 
than logarithmic asymptotic results for the complexity, using classical tools 
from orthogonal polynomials theory, i.e.~Plancherel-Rotach asymptotics for 
Hermite functions. This section is technically involved, so we restrict it to 
the study of the global complexity, i.e.~the mean number of critical points 
below a given level, and we do not push it to include the mean number of 
critical points below a given level with a fixed index. However, we remark
that at low energy
levels, the total number of critical points coincides with the
total number of local minima.  

\subsection{Acknowledgments}
We would like to thank Ivan Corwin, Percy Deift, Silvio Franz, Jorge
Kurchan, Dimitry Panchenko and Fabio Toninelli for fruitful discussions.
We want to underline our debt to Percy Deift's friendly help for the
results of Section~\ref{s:sharp}. The second author wants to thank Silvio Franz and
Jorge Kurchan for their help during the long process of sorting out and
explaining the relevant physics results. A variant of our approach has
also been used in a prior work by Fyodorov \cite{fyodorov-2004-92} and Fyodorov and Williams \cite{FW07}. We want to thank Jean-Philippe Bouchaud for
mentioning it, and Yan Fyodorov for a useful conversation. The first two
authors were partially supported by NSF Grant OISE-0730136 and by NSF Grant DMS 0806180. We thank MSRI and IMPA who gave the opportunity to give a course on the topic given here. A more pedagogical account of this subject that includes this work should appear in the MSRI  publications series. 
%>>>

\section{Notations and main results}%<<<
\label{s:notation}

We first introduce the $p$-spin spherical spin-glass model. We will fix $p$ an 
integer larger or equal to 2 (the case $p=2$ is rather trivial regarding our 
  complexity questions, it will be discussed below only  in Remark 
  \ref{remarkp2}).

A configuration $\boldsymbol\sigma$ of the $p$-spin spherical spin-glass model 
is a vector of $\mathbb R^N$ satisfying the spherical constraint
\begin{equation}
  \frac{1}{N}\sum_{i=1}^N \sigma_i^2 = 1.
\end{equation}
Thus the state space of the $p$-spin spherical spin-glass model is 
$S^{N-1}(\sqrt  N)\subset \mathbb R^N$, the Euclidean sphere of radius 
$\sqrt N$.

The Hamiltonian of the model is  the random function defined on 
$S^{N-1}(\sqrt N)$ by
\begin{equation}
  \label{Hamiltonian}
  H_{N,p}(\boldsymbol\sigma) = \frac{1}{N^{(p-1)/2}}
  \sum_{i_1, \dots, i_p=1}^N
  J_{i_1, \dots, i_p} \sigma_{i_1}\dots \sigma_{i_p}, \qquad 
  \boldsymbol\sigma=(\sigma_1,\dots,\sigma_N)\in S^{N-1}(\sqrt N),
\end{equation}
where $J_{i_1, \dots, i_p}$ are independent centered standard Gaussian random
variables. 

Equivalently, $H_{N,p}$ is the centered Gaussian process on the sphere 
$S^{N-1}(\sqrt N)$ whose covariance is given by
\begin{equation}
  \mathbb E\big[H_{N,p}(\boldsymbol\sigma)H_{N,p}(\boldsymbol\sigma')\big]
  =N^{1-p}\Big(\sum_{i=1}^N \sigma_i\sigma'_i\Big)^p 
  = N R(\boldsymbol\sigma, \boldsymbol\sigma')^p,
\end{equation}
where the normalized inner product 
$R(\boldsymbol\sigma, \boldsymbol\sigma') = 
\frac{1}{N}\< \boldsymbol\sigma, \boldsymbol\sigma'\>=\frac{1}{N}\sum_{i=1}^N \sigma_i\sigma'_i$ 
is usually called the overlap of the configurations $\boldsymbol\sigma$ 
and $\boldsymbol\sigma'$.

\medskip

We now want to introduce the complexity of spherical spin glasses. For any
Borel set $B\subset \R$ and integer $0\le k < N$, we consider the (random) number
$\Crt_{N,k}(B)$ of critical values of the
Hamiltonian $H_{N,p}$ in the set $NB=\{Nx:x\in B \}$  with index equal to
$k$,
\begin{equation}
  \label{defWk}
  \Crt_{N,k}(B) = 
  \sum_{\boldsymbol\sigma: \nabla H_{N,p}(\boldsymbol\sigma) = 0 } 
  \indi\{ H_{N,p}(\boldsymbol\sigma) \in NB\} \indi\{
    i(\nabla^2 H_{N,p}(\boldsymbol\sigma)) = k\}.
\end{equation}
Here $\nabla$, $\nabla^2$ are the gradient and the Hessian restricted to 
$S^{N-1}(\sqrt N)$, and  $i(\nabla^2 H_{N,p}(\boldsymbol\sigma))$ is the index of 
$\nabla^2 H_{N,p}$ at $\boldsymbol\sigma$, that is the number of negative eigenvalues 
of the Hessian $\nabla^2 H_{N,p}$.  We will also consider the (random) total number 
$\Crt_{N}(B)$ of critical values of the Hamiltonian $H_{N,p}$ in the set $NB$ 
(whatever their index)
\begin{equation}
  \label{e:defW}
  \Crt_{N}(B) 
  = \sum_{\boldsymbol\sigma: \nabla H_{N,p}(\boldsymbol\sigma) = 0 } 
  \indi\{ H_{N,p}(\boldsymbol\sigma) \in NB\}.
\end{equation}

Our results will give exact formulas and asymptotic estimates for the mean values 
$ \E( \Crt_{N,k}(B))$ and $\E(\Crt_{N}(B))$, when $N\to\infty$ 
and $B$, $k$ and $p$ are fixed. In particular, we will compute 
$\lim \frac{1}{N} \log \E \Crt_{N,k}(B)$ and  
$\lim \frac{1}{N} \log \E \Crt_{N}(B)$  as $N$ tends to infinity.

\medskip

Before giving the central identity relating the GOE to the complexity of 
spherical spin-glass models, we fix our notations for the GOE.

The GOE ensemble is a probability measure on the space of  real symmetric 
matrices. Namely, it is the probability distribution of the  $N \times N$ real 
symmetric random matrix $M^N$, whose entries $(M_{ij}, i\leq j)$ are 
independent centered Gaussian random variables with variance
\begin{equation}
  \label{e:Ms}
  \mathbb E M_{ij}^2 = \frac{1+\delta_{ij}}{2N}.
\end{equation}
We will denote by  $\E_\GOE = \E^N_{\GOE}$ the expectation under the GOE ensemble of
size $N\times N$.

Let $\lambda^N_0 \leq \lambda^N_1\le\dots\leq \lambda^N_{N-1}$ be the ordered 
eigenvalues of $M^N$. We will denote by 
$L_N= \frac{1}{N} \sum_{i=0}^{N-1} \delta_{\lambda^N_i}$ the (random) spectral 
measure of $M^N$, and by $\rho_N(x)$ the density of the (non-random) 
probability measure $\E_{\GOE}(L_N)$. The function $\rho_N(x)$ is usually 
called the (normalized) one-point correlation function and satisfies
\begin{equation}
  \label{e:rhonomr}
  \int_{\mathbb R} f(x)\rho_N(x)dx = \frac{1}{N} \,
  \E^N_{\GOE}\Big[\sum_{i=0}^{N-1} f(\lambda^N_i)\Big].
\end{equation}

We now state our main identity
\begin{theorem}
  \label{t:exactk}
  The following identity holds for all $N$, $p\geq 2$,  
  $k\in\{0,\dots,N-1\}$, and for all Borel sets $B \subset \mathbb R$,
  \begin{equation}
    \label{e:exactk}
    \E [\Crt_{N,k}(B)]
    = 2\sqrt\frac 2p (p-1)^{\frac N2} \E^N_{\GOE} \bigg[ e^{-N
        \frac{p-2}{2p} (\lambda^N_{k})^2} 
      \indi\Big\{ \lambda^N_{k} \in \sqrt{\frac{p}{2(p-1)}}B \Big\}
      \bigg].
  \end{equation}
\end{theorem}

Summing the preceding identities for $0\leq k\leq N-1$, we easily find the 
mean total number of critical points given a level of energy and relate it to the 
one-point function.

\begin{theorem} 
  \label{t:exactglobal}
  The following identity holds for all $N$, $p\geq 2$, and for all
  Borel sets $B\subset \mathbb R$,
  \begin{equation}
    \label{e:final}
    \E [\Crt_{N}(B)]=
    2N\sqrt\frac 2p (p-1)^{\frac N2}
    \int_{\sqrt{\frac{p}{2(p-1)}}B}
    \exp\Big\{-\frac{N(p-2)x^2}{2p} \Big\}
    \rho_N(x)\,\d x .
  \end{equation}
\end{theorem}

\begin{remark}
  \label{remarkp2}
  We want to discuss here very briefly the trivial case where $p=2$. This 
  case is easier to understand since there is a simpler connection 
  between the $2$-spin spherical model and the random matrix theory. In fact, 
  if $M$ is a $N \times N$ GOE matrix, the Hamiltonian of the $2$-spin 
  spherical model can be viewed as the quadratic form defined by the 
  symmetric matrix $M$ restricted to the sphere $S^{N-1}(\sqrt{N})$,
  \begin{equation}
    H_{N,2}(x) = (Mx, x).
  \end{equation}
  Therefore, if $p=2$, the $2N$ critical points of the Hamiltonian are simply the 
  eigenvectors of $M$, while the critical values are the eigenvalues 
  of $M$. Our 
  formulas \eqref{e:exactk}, \eqref{e:final} simplify greatly for $p=2$ and are compatible with this obvious observation.
 Indeed, for $p=2$, these formulas read
  
  \begin{equation}
    \label{e:exactkp=2}
    \E [\Crt_{N,k}(B)]
    = 2  P^N_{\GOE} [  \lambda^N_{k} \in B ]
  \end{equation}
and 
   \begin{equation}
    \label{e:finalp=2}
    \E [\Crt_{N}(B)]= 2N \rho_N(B).
  \end{equation} 
\end{remark}

We are now in a position to give our main results about the asymptotic complexity of 
spherical spin-glass models. We will see that the following number is an 
important threshold
\begin{equation}
  \label{Ec}
  E_{\infty}=E_{\infty}(p) = 2\sqrt{\frac{p-1}{p}}.
\end{equation}
Let  $I_1:(-\infty,-E_{\infty}]\to \mathbb R$ be given by
\begin{equation}
  \label{I_1}
  I_1(u)= \frac 2{E_{\infty}^2}\int_{u}^{-E_{\infty}} (z^2-E_{\infty}^2)^{1/2} \d z 
  = -\frac{u}{E_{\infty}^2}\sqrt{u^2-E_{\infty}^2} 
  - \log \Big(-u + \sqrt{u^2-E_{\infty}^2} \Big) 
  +  \log E_{\infty}.
\end{equation}
\begin{remark}
  \label{r:Is}
  In \cite{BDG01}, it is proved that $I_1(u)$ is the rate function of the LDP 
  for the smallest eigenvalue of the GOE with the proper normalization of the 
  variance of the entries (more precisely, on its domain $I_1(u)=I_1(-u;E_{\infty}/2)$, 
  see \eqref{e:LDPrate}).
\end{remark}

We now define the following important functions which will describe the 
asymptotic complexity of the $p$-spin spherical spin-glass models.
\begin{equation}
  \label{e:complexity}
  \Theta_p(u)  =
  \begin{cases}
    \frac{1}{2}\log
    (p-1) -   \frac{p-2}{4(p-1)}u^2 - I_1(u),
    &\text{if } u \leq -E_{\infty},\\
    \frac{1}{2}\log
    (p-1) -    \frac{p-2}{4(p-1)}u^2,
    &\text{if } -E_{\infty} \leq u \leq 0, \\
    \frac{1}{2}\log(p-1),
    & \text{if }  0 \leq u, 
  \end{cases}
\end{equation}
and, for any integer  $k \geq 0$,
\begin{equation}
  \label{e:thetakp}
  \Theta_{k,p}(u) =
  \begin{cases}
    \frac{1}{2}\log (p-1) 
    -   \frac{p-2}{4(p-1)}u^2 - (k+1) I_1(u),
    &\text{if } u \leq -E_{\infty},\\
    \frac{1}{2}\log (p-1) 
    -  \frac{p-2}{p}, & \text{if }  u\ge -E_{\infty}.
  \end{cases}
\end{equation}
We note that $\Theta_p(u), \Theta_{k,p}(u)$ are non-decreasing, continuous 
functions on $\R$, with maximal values $\frac{1}{2} \log(p-1)$, 
$\frac{1}{2} \log(p-1) - \frac{p-2}{p}$, respectively (see
  Figure~\ref{f:thetas}). 

We now give the logarithmic asymptotics of the complexity of spherical spin glasses. 
To simplify the statement, we fix  $B = (-\infty,u)$, $u\in \mathbb R$,
and we write $\Crt_{N,k}(u)=\Crt_{N,k}(B)$, $\Crt_{N}(u)=\Crt_{N}(B)$.
\begin{theorem}
  \label{t:complexityk}
  For all $ p \geq 2$ and $k \geq 0$ fixed, 
  \begin{equation}
    \label{critical2}
    \lim_{N\to\infty}
    \frac{1}{N} \log \E \Crt_{N,k}(u) = \Theta_{k,p}(u).
  \end{equation}
\end{theorem}

\begin{figure}[ht]
  \includegraphics[width=10cm]{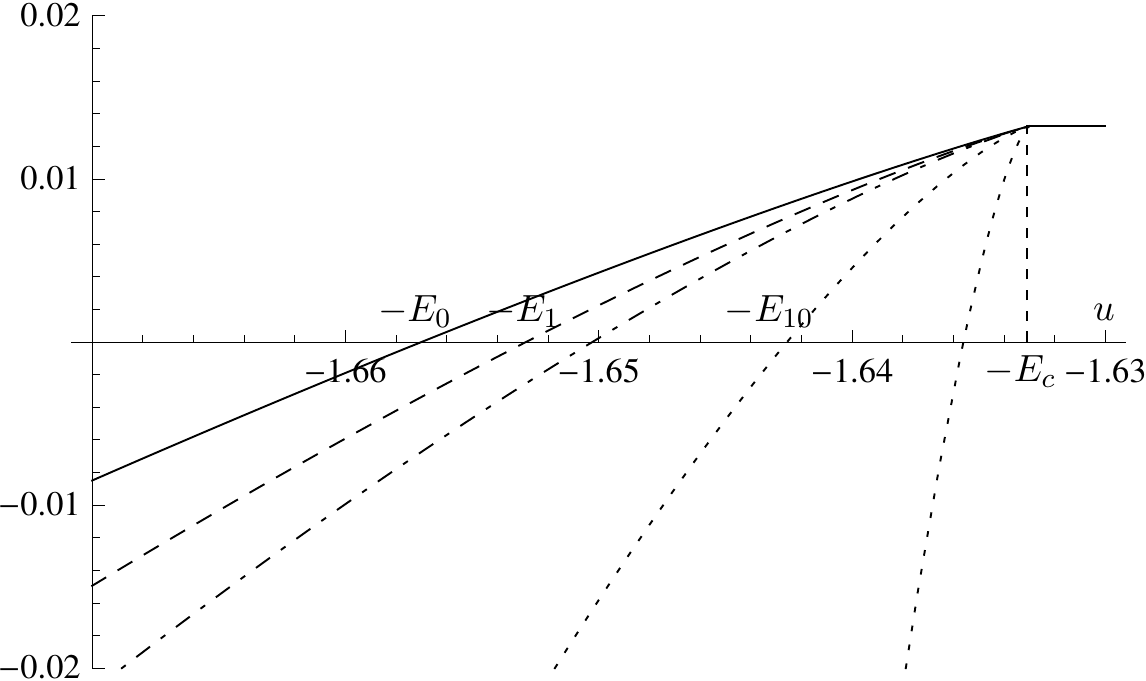}
  \caption{The functions $\Theta_{k,p}$ for $p=3$ and $k=0$ (solid), $k=1$
  (dashed), $k=2$ (dash-dotted), $k=10$, $k=100$ (both dotted). All these
  functions agree for $u\ge -E_{\infty}$.} 
  \label{f:thetas}
\end{figure}

\begin{remark}
  \label{r:symmetry1} 
  It is straightforward to extend the last theorem to
  general Borel sets $B$ (see Remark \ref{r:remextension}). Furthermore,
  by symmetry, Theorem \ref{t:complexityk} also holds as stated for the
  random variables $\Crt_{N,N-l}((u,\infty))$, with $l \geq 1$ fixed, if one
  replaces $\Theta_{k,p}(u)$ by $\Theta_{l-1,p}(-u)$. 
\end{remark}  
\begin{remark} 
  \label{r:symmetry} 
  For the local 
  minima, i.e.~when $k=0$, the limit formula given by Theorem 
  \ref{t:complexityk} is precisely the formula given by physicists in 
  \cite{CS95}, \cite{CLR03}. Arguing via a TAP approach (to be 
    described below in Section~\ref{s:TAP}), they derive the following 
  asymptotic complexity of local minima,
  \begin{equation}
    g(E)=\frac{1}{2}
    \Big\{\frac{2-p}{p} - \log \Big(\frac{p z^2}{2}\Big)
      +\frac{p-1}{2}z^2 - \frac{2}{p^2z^2}\Big\},
  \end{equation}
  where $z=\frac{1}{p-1}\big(-E-(E^2-\frac{2(p-1)}{p})^{1/2}\big)$. In 
  Section~\ref{s:TAP}, we show that, in fact, 
  $g(E)=\Theta_{0,p}(2^{-1/2}E)$. The factor $2^{-1/2}$ comes from 
  the fact that in \cite{CS95} the Hamiltonian $H$ has a different 
  normalization.
\end{remark}

We also provide an exponential asymptotic for the expected total number of 
critical values below level $Nu$.

\begin{theorem}
  For all $ p \geq 2$,
  \label{t:complexityglobal}
  \begin{equation}
    \label{critical}
    \lim_{N\rightarrow\infty}
    \frac{1}{N} \log \E \Crt_{N}( u) =  \Theta_p(u).
  \end{equation}
\end{theorem}

\begin{remark}
  As a simple consequence of Theorems \ref{t:complexityk} and 
  \ref{t:complexityglobal}, one can easily compute the logarithmic asymptotics of 
  the mean total number of critical points 
  $\E(\Crt_{N}(\mathbb R))$  and the mean total number of 
  critical points of index $k$, 
  $\E(\Crt_{N,k}(\mathbb R))$.
  \begin{align}
    &\lim_{N\rightarrow\infty}
    \frac{1}{N} \log  \E(\Crt_{N}(\mathbb R))=\frac{1}{2}\log(p-1) ,\\
    &\lim_{N\rightarrow\infty}
    \frac{1}{N} \log  \E(\Crt_{N,k}(\mathbb R))= \frac{1}{2}\log(p-1)-
    \frac{p-2}{p} .
  \end{align}
  This agrees with formula (13) of \cite{CS95}. Note the fact that the mean number of critical points of index k is independent of k (at least in these logarithmic estimates)
\end{remark}
\begin{remark}
  Theorems $\ref{t:complexityk}$ and $\ref{t:complexityglobal}$ are simple 
  consequences of large deviation properties for random matrices, using the 
  results of \cite{BDG01} and of \cite{BG97}. In Appendix~\ref{a:LDPsection}, 
  we recall the LDP for the empirical spectral measure of the GOE
  proved in \cite{BG97}, and we prove a LDP for the $k$-th 
  largest eigenvalue of a GOE matrix, extending the results of \cite{BDG01}.
\end{remark}

\begin{remark}
  \label{remarkp2bis}
  The case $p=2$ is particular since the total complexity is then always non positive.
  But for any $p\geq3$ the complexity is positive.
\end{remark}

Using our results about complexity, we now want to  extract some 
information about the geometry of the bottom of the energy landscape $H_{N,p}$. 
For any integer $k \geq 0$, we introduce $E_k = E_k(p)>0$ as the unique 
solution to (see Figure~\ref{f:thetas} again).
\begin{equation}
  \label{e:E_k}
  \Theta_{k,p}(-E_k(p)) = 0.
\end{equation}

These numbers will be crucial in the description of the ground state and 
of the low-lying critical values of the Hamiltonian $H_{N,p}$. It is 
important to note that, for any fixed $p \geq 3$, the sequence 
$(E_k(p))_{k\in \N}$ is strictly decreasing, and converges  to $E_{\infty}(p)$ as 
$k\to\infty$. The first result we want to derive is about the 
ground state energy, which we define as the (normalized) minimum of the 
Hamiltonian  $H_{N,p}$
\begin{equation}
  \label{GroundState}
  GS^N = \frac{1}{N} \inf_{\boldsymbol\sigma \in S^{N-1}(\sqrt N)}
  H_{N,p}(\boldsymbol\sigma).
\end{equation}

\begin{theorem}\label{t:groundstate}
  For every  $p\ge 3$ $$\liminf_{N\to\infty}GS^N\ge -E_0(p)$$
 Moreover for $p\ge 4$ even,
  \begin{equation}
    \lim_{N\to\infty}GS^N=-E_0(p)\qquad\text{in probability}.
  \end{equation}
\end{theorem}
\begin{remark}
  The lower bound  on the Ground State follows from our 
  complexity estimates and holds for all $p\ge 3$. To obtain a matching upper 
  bound we use the Parisi formula and the one step replica symmetry breaking as proved by Talagrand~\cite{Tal06}. The Parisi formula is proven there for every $p$ but the one step replica symmetry breaking is proven only for even $p$'s. It might be noteworthy that if we could go beyond our annealed estimates of the complexity we would be in a position to prove directly the one-step replica symmetry breaking at zero temperature 
  \end{remark}

By Theorem~\ref{t:groundstate}, it is improbable to find a critical value below 
the level $-NE_0(p)$. The next interesting phenomenon is the role of the 
threshold $E_{\infty}(p)$. Namely, it is (even more) improbable to find, above the 
threshold $-NE_{\infty}(p)$, a critical value of the Hamiltonian of a fixed index $k$, 
when $N\to\infty$. Otherwise said, above the threshold $-N E_{\infty}(p)$, all 
critical values of the Hamiltonian must be of diverging index, with 
overwhelming probability. 

\begin{theorem}\label{t:nofiniteindex}
  Let for an integer $k\ge 0$ and $\varepsilon >0$, $B_{N,k}(\varepsilon )$ be the 
  event ``there is a critical value of index $k$ of the Hamiltonian $H_{N,p}$ 
  above the level $-N(E_{\infty}(p)-\varepsilon)$'',  that is
  $B_{N,k}(\varepsilon )=\{\Crt_{N,k}((-E_{\infty}(p)+\varepsilon ,\infty))>0\}$. Then for 
  all $k\ge 0$ and $\varepsilon >0$,
  \begin{equation}
    \limsup_{N\rightarrow\infty}
    \frac{1}{N^2} \log \Pro (B_{N,k}(\varepsilon)) < 0.
  \end{equation}
\end{theorem}

By the last theorem, all critical values of the Hamiltonian of fixed index (non 
  diverging with $N$) must be found in the band $(-NE_0(p), -NE_{\infty}(p))$. We now 
explain the role of the thresholds $E_k(p)$. Namely, it is improbable to find 
critical value of index larger or equal to $k$ below the threshold $-NE_{k}(p)$, 
for any fixed integer $k$. 

\begin{theorem}
  \label{t:belowEk}
  For $k\geq 0$ and  $\varepsilon > 0$, let $A_{N,k}(\varepsilon)$ to be the 
  event ``there is a critical value  of the Hamiltonian $H_{N,p}$ below the 
  level $-N(E_{k}(p)+\varepsilon)$ and with index larger or equal to $k$'', 
  that is 
  $A_{N,k}(\varepsilon )=\{\sum_{i=k}^\infty \Crt_{N,i}(-E_k(p)-\varepsilon )>0\}$. 
  Then for all $k\ge 0$ and $\varepsilon >0$,
  \begin{equation}
    \limsup_{N\rightarrow\infty}
    \frac{1}{N} \log \Pro (A_{N,k}(\varepsilon)) < 0.
  \end{equation}
\end{theorem}

Theorem \ref{t:belowEk} describes an interesting layered structure for the 
lowest critical values of the Hamiltonian $H_{N,p}$. It says that the 
lowest critical  values above the ground state energy (asymptotically 
  $-NE_0(p)$) are (with an overwhelming probability again) only local 
minima, this being true up to  the value $-NE_1(p)$, and  that in a layer 
above, $(-NE_1(p), -NE_2(p))$, one finds only critical values with index 0 
(local minima) or saddle point with index 1, and above this layer one 
finds only critical values with index 0,1 or 2, etc. 
This  picture was already predicted by 
physicists for minima \cite{CS95}, \cite{CLR03} and for 
critical points of finite indices \cite{KL96}. In particular, this says 
that the energy barrier to cross when starting from the ground state in 
order to reach another local minima diverges with $N$, since it is bounded 
below by the energy difference between an index-one saddle point and the 
ground state, i.e.~by $N(E_0(p) -E_1(p))$.

\begin{remark}
  Even though it does not follow immediately from Theorem \ref{t:belowEk} 
  and from our results on complexity, it is tempting to conjecture that 
  the minimum possible energy of a critical point of index $k$, normalized 
  by $N$, should converge to $-E_{k}(p)$ (For $k=0$ this is the statement 
    of Theorem \ref{t:groundstate}), while likewise the maximum energy of a critical point 
  of index $k$, once normalized by $N$, should converge to $-E_{\infty}$. It is 
  also tempting to conjecture that the main contribution to the number of 
  critical points of a finite index $k$ is given by those whose energy is 
  asymptotically $-NE_{\infty}$. That is, the number of critical points of any finite 
  index with energy strictly below $-NE_{\infty}$ should be negligible with 
  respect to those with energy near $-NE_{\infty}$ (with probability going to one, 
    as $N$ tends to infinity). However, near any energy value in 
  $E \in (-E_k(p),-E_{\infty})$ there are still an exponentially many 
  critical values of index $k$. We cannot  reach those statements at this point because 
  our complexity results concern only the first moment of 
  $\Crt_{N,k}(u)$. We would need to control the concentration of 
  these random variables.
\end{remark}

In Section~\ref{s:sharp}, we show that the precision of Theorem
\ref{t:complexityglobal} can
be improved and we derive, using asymptotic properties of orthogonal
polynomials, the following sharp asymptotics of $\E (\Crt_{N}(u))$. 
\begin{theorem}
  \label{t:sharp}
  For $p\ge 3$, the following holds as $N\to\infty$:
  \begin{itemize}
    \item[(a)] For $u<-E_{\infty}$
    \begin{equation}
      \label{e:tsharpexp}
      \E \Crt_{N}(u) =   
      \frac { h(v)}{( 2p\pi )^{1/2} }
      \frac{e^{I_1(v,\frac{1}{\sqrt{2}}) (v)- \frac v2I_1(v,\frac{1}{\sqrt{2}})' (v)}}
      {-\phi'(v) + I_1(v,\frac{1}{\sqrt{2}})'(v)}\,
      N^{-1/2}e^{N\Theta_p(u)}(1 + o(1)),
    \end{equation}
    where $v=-u\sqrt{\frac p{2(p-1)}}$ and  the functions $h$, $\phi$ and $I_1(v,\frac{1}{\sqrt{2}})$ are 
    given in \eqref{e:CvF}, \eqref{e:psis}.

    \item[(b)] For $u=-E_{\infty}$
    \begin{equation}
      \label{e:tsharpairy}
      \E \Crt_{N}( -E_{\infty}) = 
      \frac{2\Ai(0) \sqrt{2 p}}{3(p-2)} N^{-1/3} 
      e^{N\Theta_p(-E_{\infty})}  (1 + o(1)).
    \end{equation}

    \item[(c)] For $u\in (-E_{\infty},0)$
    \begin{equation}
      \label{e:tsharpbulk}
      \E \Crt_{N}( u) = 
      \frac{2\sqrt{2p(E_{\infty}^2-u^2)}}{(2-p)\pi u} e^{N\Theta_p(u)}
      (1+o(1)).
    \end{equation}

    \item[(d)] For $u>0$
    \begin{equation}
      \label{e:tsharpneg}
      \E \Crt_{N}(u) = 2 \E \Crt_{N}(0)(1+o(1))
      = \frac {4\sqrt 2}{\sqrt{\pi(p-2)}}N^{1/2} e^{ N \Theta_p(0)}  
       (1 + o(1)).
    \end{equation}
  \end{itemize}
\end{theorem}
Since $\Theta_k(u)<\Theta_0(u)$  for all $k>0$ and $u<-E_{\infty}$,
we obtain as an easy consequence of Theorem~\ref{t:complexityk} 
the following sharp asymptotics for the mean number of minima.
\begin{corollary}
  \label{c:sharp}
    For $u<-E_{\infty}$, 
    
    \begin{equation}
     \E \Crt_{N,0}(u) =   
      \frac { h(v)}{( 2p\pi )^{1/2} }
      \frac{e^{I_1(v,\frac{1}{\sqrt{2}}) (v)- \frac v2I_1(v,\frac{1}{\sqrt{2}})' (v)}}
      {\phi'(v) + I_1(v,\frac{1}{\sqrt{2}})'(v)}\,
      N^{-1/2}e^{N\Theta_p(u)}(1 + o(1)),
    \end{equation}
    
    where $v=-u\sqrt{\frac p{2(p-1)}}$ and  the functions $h$, $\phi$ and $I_1(v,\frac{1}{\sqrt{2}})$ are 
    given in \eqref{e:CvF}, \eqref{e:psis}.
     
\end{corollary}
%>>>

\section{Proof of the central identity}%<<<
\label{s:exact}

In this section we prove Theorems \ref{t:exactk} and \ref{t:exactglobal}.
For the proofs, we find it more convenient to work with processes of 
variance one on the unit sphere $S^{N-1}\subset \mathbb R^N$ rather than 
to work with $H_{N,p}$. Hence, for $\boldsymbol\sigma\in S^{N-1}$ we define,
\begin{equation}
  \label{definitionf}
  f_{N,p}(\boldsymbol\sigma) = \frac{1}{\sqrt{N}} H_{N,p}(\sqrt{N}\boldsymbol\sigma).
\end{equation}
We will regularly omit the subscripts $N$ and $p$ to save on notations, 
$f=f_{N,p}$. The function~$f$ is again centered Gaussian process whose
covariance satisfies
\begin{equation}
  \label{eq:overlap}
  \E(f(\boldsymbol\sigma) f(\boldsymbol\sigma')) 
  = \Big(\sum_{i=1}^N \sigma_i \sigma'_i\Big)^p 
  = R(\boldsymbol\sigma, \boldsymbol\sigma')^p,\qquad
  \boldsymbol \sigma , \boldsymbol \sigma '\in S^{N-1}.
\end{equation}

To estimate the mean number of critical points in a certain level we will 
use the \textit{Kac-Rice} formula as it appears in the recent book of Adler 
and Taylor \cite{AT07} which we now formulate as a lemma. We use $\<x,y\>$ to 
denote the usual Euclidean scalar product, as well as the scalar product on any 
tangent space $T_{\boldsymbol\sigma}S^{N-1}$. Let $\nabla^2 f$ be the covariant 
Hessian of $f$ on $S^{N-1}$ defined, e.g., by 
$\nabla^2 f(X,Y)=XYf - \nabla_X Y f$. Here $\nabla_X Y$ is the usual Riemann 
connection and $X,Y\in T S^{N-1}$ are tangent vectors. On $S^{N-1}$ we fix an 
arbitrary orthonormal frame field $(E_i)_{1\le i <N}$, that is a set of 
$N-1$ vector fields $E_i$ on $S^{N-1}$ such that $\{E_i(\boldsymbol\sigma)\}$ 
is an orthonormal basis of $T_{\boldsymbol\sigma} S^{N-1}$. We write 
$\phi_{\boldsymbol\sigma}$ for the density of the gradient vector 
$(E_i f(\boldsymbol \sigma))_{1\le i<N}$ and 
$\det  \nabla^2f(\boldsymbol\sigma)$ for the determinant of the matrix 
$(\nabla^2 f(E_i,E_j)(\boldsymbol\sigma))_{1\le i,j<N}$.

\begin{lemma}
  \label{l:AT}
  Let $f$ be a centered Gaussian field on $S^{N-1}$ and let 
  $\mathcal A=(U_\alpha ,\Psi_\alpha )_{\alpha \in I}$ be a finite atlas 
  on $S^{N-1}$. Set 
  $f^\alpha = f \circ \Psi_\alpha ^{-1}: \Psi_\alpha  (U_\alpha )\subset \mathbb R^{N-1}\to \mathbb R$ 
  and define $f^\alpha_i=\partial f^\alpha /\partial x_i$, 
  $f^\alpha_{ij}=\partial^2 f^\alpha /\partial x_i\partial x_j$. Assume 
  that for all $\alpha \in I$ and all $x,y\in \Psi_\alpha  (U_\alpha )$ the joint distribution of 
  $(f^\alpha_i(x),f^\alpha_{ij}(x))_{1\le i \le j< N}$ is 
  non-degenerate, and
  \begin{equation}
    \label{e:logcond}
    \max_{i,j}
    \big|\Var(f^\alpha_{ij}(x))+\Var(f^\alpha_{ij}(y))
    -2\Cov(f^\alpha _{ij}(x),f^\alpha_{ij}(y))\big|
    \le K_\alpha  |\ln |x-y||^{-1-\beta }
  \end{equation}
  for some $\beta >0$ and $K_\alpha >0$. For a Borel set
  $B\subset \mathbb R$, let
  \begin{equation}
    \Crt_{N,k}^f(B) 
    = \sum_{\boldsymbol\sigma: \nabla f(\boldsymbol\sigma) = 0 } 
    \indi\{i(\nabla^2 f(\boldsymbol \sigma ))=k, f(\boldsymbol\sigma) \in B \}.
  \end{equation} 
  Then, 
  using $\d \boldsymbol\sigma$ to denote the usual surface measure on
  $S^{N-1}$, 
  \begin{align}
    \label{e:metak}
    \E \Crt_{N,k}^f(B)&=
    \int_{S^{N-1}}
    \E \big[ | \det \nabla^2f(\boldsymbol\sigma) |
      \indi\{f(\boldsymbol\sigma) \in B,i(\nabla^2 f(\boldsymbol \sigma ))=k\}\, \big|\,
      \nabla f(\boldsymbol\sigma) = 0 \big]
    \phi_{\boldsymbol\sigma}(0) \d \boldsymbol\sigma.
  \end{align}
\end{lemma}

\begin{proof}
  Assumptions of the lemma, which are taken from Corollaries~11.3.2 and 
  11.3.5 of \cite{AT07}, assure that $f$ is a.s.~a Morse function and its 
  gradient and Hessian exist in $L^2$ sense.   The lemma can be then 
  proved using the same procedure as Theorem~12.4.1 of \cite{AT07}.  Our 
  formula~\eqref{e:metak} is analogous to the display just following 
  formula (12.4.4) of \cite{AT07}, modulo the term $(-1)^k$ which is 
  missing in our settings since we are interested in the number of 
  critical points of $f$ in $B$  and not in the Euler characteristic of
  the excursion set.
\end{proof}

An application of Lemma~\ref{l:AT} is made possible due to the following
lemma which describes the joint law of Gaussian vector $(f(\boldsymbol
    \sigma), \nabla f(\boldsymbol \sigma), \nabla^2
  f(\boldsymbol\sigma)))$. A similar computation is also present in section 8.2 of \cite{AW09}.

\begin{lemma}
  \label{l:conditioning}
  (a) Let $(f_{i}(\boldsymbol \sigma))_{1\leq i < N} $ be the gradient and  
  $(f_{ij}(\boldsymbol \sigma ))_{1\le i,j<N}$ the Hessian matrix at 
  $\boldsymbol \sigma \in S^{N-1} $, that is 
  $f_{i}= E_i f(\boldsymbol \sigma ),f_{ij} =\nabla^2 f(E_i,E_j)(\boldsymbol \sigma )$. 
  Then, for all $1\le i,j,k< N$, $f(\boldsymbol \sigma )$, 
  $f_i(\boldsymbol \sigma )$, $f_{jk}(\boldsymbol \sigma )$ are centered 
  Gaussian random variables whose joint distribution is determined by
 \begin{equation}
    \label{e:covariancesa}
    \begin{aligned}
      &\mathbb E[f(\boldsymbol \sigma)^2]=1,\\
      &\mathbb E[f(\boldsymbol \sigma) f_{ij}(\boldsymbol \sigma)]=-p \delta_{ij},\\
    \end{aligned}
    \qquad
    \begin{aligned}
      &\mathbb E[f(\boldsymbol \sigma)f_i(\boldsymbol \sigma)]=
      \mathbb E[f_i(\boldsymbol \sigma)f_{jk}(\boldsymbol \sigma)]=0,\\
      &\mathbb E[f_i(\boldsymbol \sigma)f_j(\boldsymbol \sigma)]=p \delta_{ij},\\
    \end{aligned}
  \end{equation}
  and
  \begin{equation}
    \label{e:covariancesb}
    \mathbb E[f_{ij}(\boldsymbol \sigma)f_{kl}(\boldsymbol \sigma)]=
    p(p-1)(\delta_{ik}\delta_{jl}+\delta_{il}\delta_{jk})+
    p^2 \delta_{ij}\delta_{kl}.
  \end{equation}

  (b) Under the 
  conditional distribution $\mathbb P[\cdot | f(\boldsymbol \sigma)=x]$,
  $x\in \mathbb R$, the random variables $ f_{ij}(\boldsymbol \sigma )$, 
  $1\le i,j< N$, are independent Gaussian variables satisfying
  \begin{equation}
    \begin{split}
      \mathbb E[f_{ij}(\boldsymbol \sigma )]&=-xp \delta_{ij},\\
       \E \big[ f_{ij}(\boldsymbol  \sigma )^2\big]&=(1+\delta_{ij})p(p-1).
    \end{split}
  \end{equation}
  Alternatively,
  the random matrix $(f_{ij}(\boldsymbol \sigma ))$ has the same distribution as 
  \begin{equation}
    M^{N-1}\sqrt {2(N-1)p(p-1)}-xp I,
  \end{equation}
  where $M^{N-1}$ is
  $(N-1)\times(N-1)$ GOE matrix given by \eqref{e:Ms} and $ I$
  is the identity matrix.
\end{lemma}
\begin{proof}
  Without loss of generality we can suppose that $\boldsymbol \sigma $ is the north
  pole of the sphere $\boldsymbol n=(0,\dots,0,1)$. We 
  define the function 
  $\Psi :S^{N-1}\to \mathbb R^{N-1}$ by 
  $\Psi (x_1,\dots,x_N)=(x_1,\dots,x_{N-1})$. It is a chart in some 
  neighborhood $U$ of $\boldsymbol n$. 
  We set
  \begin{equation}
    \label{e:fbar}
    \bar f = f\circ \Psi^{-1},
  \end{equation}
  which is a Gaussian process on $\Psi (U)$ with covariance
  \begin{equation}
    C(x,y)=\Cov(\bar f(x),\bar f(y))=
    \bigg\{\sum_{i=1}^{N-1} x_iy_i
      +\sqrt{\big(1- \textstyle\sum_{i=1}^{N-1} x_i^2\big)
        \big(1- \sum_{i=1}^{N-1} y_i^2\big)}\bigg\}^p.
  \end{equation}
  We choose the orthonormal frame field $(E_i)$ such that it satisfies 
  $E_i(\boldsymbol n)=\partial/\partial x_i$ with respect to the chart $\Psi$. 
  Then the covariant Hessian $(f_{ij}(\boldsymbol n))$ agrees with the usual Hessian of
  $\bar f$ at 0,  by
  noting that the Christoffel symbols $\Gamma^i_{kl}(\boldsymbol n )\equiv 0$.  
  Hence, to check \eqref{e:covariancesa}, \eqref{e:covariancesb}, we should
  prove analogous identities for $\bar f(0)$,
  $\bar f_i(0)=\frac \partial {\partial x_i} \bar f(0)$ and 
  $\bar f_{ij}(0)=\frac {\partial^2} {\partial x_i\partial x_j} \bar f(0)$.  
  
  The covariances of $\bar f, \bar f_i, \bar f_{ij}$,
  can be computed using a well-known formula (see e.g.~\cite{AT07}
    formula (5.5.4)),
  \begin{equation}
    \label{e:covariance}
    \Cov\Big(\frac {\partial^k \bar f(x)}{\partial x_{i_1}\dots\partial x_{i_k}}
      \frac {\partial^\ell \bar f(y)}{\partial y_{j_1}\dots\partial y_{j_\ell}}
      \Big)
    =
    \frac {\partial^{ k+\ell} C(x,y)}{\partial x_{i_1}\dots\partial x_{i_k}
      \partial y_{j_1}\dots\partial y_{j_\ell}}.
  \end{equation}
  Straightforward algebra then gives \eqref{e:covariancesa},
  \eqref{e:covariancesb}.
  Moreover, since the derivatives of a centered Gaussian field have centered Gaussian 
  distribution,  relations \eqref{e:covariancesa} and 
  \eqref{e:covariancesb} determine uniquely the joint distribution of 
  $f(\boldsymbol \sigma)$, $f_i(\boldsymbol \sigma)$, and
  $f_{ij}(\boldsymbol \sigma)$. This completes the proof of the claim (a).  
  
  The well-known rules how
  Gaussian distributions transform under conditioning (see, e.g., \cite{AT07}, 
    pages 10--11) then yield the claim (b) of the lemma.
\end{proof}

The next tool in order to prove Theorem~\ref{t:exactk} is the following
lemma which is an independent fact about the distribution of the
eigenvalues of the GOE. It allows us to deal with the (usually rather
  unpleasant) absolute value of the
determinant of the Hessian that appears in \eqref{e:metak}.

\begin{lemma}\label{LemmaGOEAp}
  Let $M^{N-1}$ be a $(N-1)\times (N-1)$ GOE matrix and $X$ be 
  an independent Gaussian random variable with mean $m$ and variance $t^2$.  
  Then, for any Borel set $G\subset\mathbb R$,
  \begin{equation}
    \label{e:Lemsefla}
    \begin{split}
      &\E \Big[ \big|\det M^{N-1}- X I\big| \indi\big\{ i(M^{N-1}- X I)=k,
          X \in G\big\}
        \Big] 
      \\&= \frac{\Gamma(\frac{N}{2})(N-1)^{-\frac N2}}{\sqrt{\pi t^2}}  
      \E_\GOE^N \bigg[ \exp\bigg\{\frac{N(\lambda_{k}^{N})^2}{2} 
          - \frac{\big(\{\frac{N}{N-1}\}^{\frac 12}\lambda_{k}^{N}-m\big)^2}{2t^2}\bigg\} 
        \indi\Big\{ \lambda_{k}^{N} \in \{\tfrac{N-1}{N}\}^{\frac12}G\Big\} \bigg].
    \end{split} 
  \end{equation} 
\end{lemma}

\begin{proof}
  The left-hand side of \eqref{e:Lemsefla} can be written as
  \begin{equation}\label{e:AL1}
    \frac{1}{\sqrt{2\pi t^2}}\int_{G} e^{-\frac{(x-m)^2}{2t^2}} 
    \E_{GOE}^{N-1}\Big[ \big|\det(M-xI)\big|
      \indi\big\{i(M-xI)=k\big\}\Big] \d x.
  \end{equation}
  Observe that the event $\{i(M-xI)=k\}$ is equal to the event $\{A^{N}_k (x)\}$, where $A^{N}_k(x)$ is given by
  \begin{equation}
    \label{e:bad}
    \big\{\lambda^{N-1} : \lambda_0^{N-1} \leq \dots
      \leq \lambda_{k-1}^{N-1} < x
      \le\lambda_{k}^{N-1}\le \ldots \leq
      \lambda_{N-2}^{N-1}\big\}.
  \end{equation}
  Recall the explicit formula for the distribution $Q_N$ of the eigenvalues
  $\lambda^N_i$ of the GOE matrix $M^N$, 
  (see, e.g., \cite{Meh91}, p.~519),
  \begin{equation}
    \label{e:defQ}
    Q_N(\d \lambda^N )=
    \frac{1 } {Z_N}
    \prod_{i=0}^{N-1} e^{-\frac N2 (\lambda^N_i)^2}\d \lambda^N_i
    \prod_{0\le i< j<N}|\lambda^N_i-\lambda^N_j|
    \indi\{\lambda^N_0\le \dots \le \lambda^N_{N-1}\},
  \end{equation}
  where the normalization $Z_N$ can be computed from Selberg's
  integral (cf.~\cite{Meh91}, p.~529)
  \begin{equation}
    \label{e:defZ}
    Z_N=\frac 1 {N!}(2\sqrt 2 )^{N}N^{-N(N+1)/4}
    \prod_{i=1}^N \Gamma \Big(1+\frac i2\Big).
  \end{equation}
  Using this notation, 
  \begin{equation}
  \E_{GOE}^{N-1}\Big[ \big|\det(M-xI)\big|
      \indi\big\{i(M-xI)=k\big\}\Big] = \int_{A^{N}_k(x)} \prod_{i=0}^{N-2} |\lambda^{N-1}_i-x|
    Q_{N-1}(\d\lambda^{N-1}).
  \end{equation}
  and  \eqref{e:AL1} becomes
  \begin{equation}
    \frac{1}{\sqrt{2\pi t^2}}\int_{G} e^{-\frac{(x-m)^2}{2t^2}} 
    \int_{A^{N}_k(x)} \prod_{i=0}^{N-2} |\lambda^{N-1}_i-x|
    Q_{N-1}(\d\lambda^{N-1}) \d x.
  \end{equation}
  Comparing the product in the integrand with the van der Monde determinant in 
  \eqref{e:defQ} suggests considering $x$ as the $k+1$-th smallest eigenvalue 
  of a $N\times N$ GOE matrix. Indeed, if we substitute 
  $\lambda^{N-1}_i = \{ \frac{N}{N-1}\}^{1/2} \lambda^N_i $, for 
  $i \in \{0, \ldots, k-1 \}$, 
  $\lambda^{N-1}_i = \{ \frac{N}{N-1}\}^{1/2} \lambda^N_{i+1}$ for 
  $i \in \{k, \ldots, N-2 \}$ and 
  $x = \{\frac{N}{N-1} \}^{1/2}\lambda_{k}^{N}$ and perform the change of 
  variables, then we  obtain
  \begin{equation}\label{e:defase}
    \begin{split}
      &\frac{Z_{N}}{Z_{N-1}\sqrt{2\pi t^2}} 
      \Big(\frac{N}{N-1}\Big)^{\frac{(N+2)(N+1)}{4}} \\ 
      &\times \E_\GOE^N\bigg[ \exp\bigg(\frac{N(\lambda_{k}^{N})^2}{2} - 
          \frac{(\{\frac{N}{N-1}\}^{1/2}\lambda_{k}^{N}-m)^2}{2t^2}\bigg) 
        \indi \Big\{ \lambda_{k}^{N} \in \big\{\tfrac{N-1}{N}\big\}^{1/2}G
          \Big\} \bigg].
    \end{split}
  \end{equation}
  Plugging \eqref{e:defZ} into \eqref{e:defase} completes the proof of the
  lemma.
\end{proof}

The three above lemmas yield the proof of Theorem~\ref{t:exactk} as follows.

\begin{proof}[Proof of Theorem~\ref{t:exactk}]
  We first verify the conditions of Lemma~\ref{l:AT}. We use the same chart
  and orthogonal frame  as in the proof of Lemma~\ref{l:conditioning}.
  From \eqref{e:covariancesa}, \eqref{e:covariancesb}, it is not difficult
  to check that the joint distribution of
  $(f_i(\boldsymbol \sigma ), f_{ij}(\boldsymbol \sigma ))$ is non-degenerate for
  $\boldsymbol \sigma =\boldsymbol n$. By the
  continuity of the covariances, it is then non-degenerate for all
  $\boldsymbol \sigma \in U$,
  if $U$ is small enough. Similarly, using \eqref{e:covariance} we can
  verify  that \eqref{e:logcond} is satisfied on $U$. 
  Since $S^{N-1}$ can be covered by a finite number
  of copies of $U$ obtained by rotations around the center of sphere,
  the conditions of Lemma~\ref{l:AT} are satisfied.

  We can thus apply formula \eqref{e:metak}. First note that, due to the 
  rotational symmetry again, the integrand does not depend on 
  $\boldsymbol\sigma$. Hence, recalling that $\boldsymbol n$ denotes the 
  north pole and using \eqref{e:defW},\eqref{definitionf},
  \begin{equation}
    \label{e:baa}
    \begin{split}
      &\mathbb E \Crt_{N,k}(B) \\ &= \omega_N
      \E \Big[ \left| \det \nabla^2f(\boldsymbol n)   \right|
        \indi\{i(\nabla^2f_N(\boldsymbol\sigma))=k\}
        \indi\{f(\boldsymbol n) \in \sqrt{N}B\} \Big|
        \nabla f(\boldsymbol n) = 0 \Big]
      \phi_{\boldsymbol n }(0),
    \end{split}
  \end{equation}
  where $\omega_N$ is the volume of the sphere $S^{N-1}$,
  \begin{equation}
    \label{e:aab}
    \omega_N=\frac{2\pi^{N/2}}{\Gamma (N/2)}.
  \end{equation}
  The density of 
  gradient $\nabla f(\boldsymbol n)$ is same as the density of $(\bar f_i(0))_{1\le i<N}$. 
  Hence, using \eqref{e:covariancesa},
  \begin{equation}
    \label{e:aac}
    \phi_{\boldsymbol n}(0)= (2 \pi  p)^{-(N-1)/2}.
  \end{equation}
  To compute the expectation in \eqref{e:baa}, we condition on
  $f(\boldsymbol n)$ and use
  the fact that, by \eqref{e:covariancesa}, both $f$ and its Hessian are
  independent of 
  the gradient,
  \begin{equation}
    \label{e:bab}
    \begin{split}
      \E &\bigg[ | \det \nabla^2f_N(\boldsymbol n) |
        \indi\{i(\nabla^2f_N(\boldsymbol\sigma))=k, f(\boldsymbol n) \in \sqrt{N}B\} \big|
        \nabla f(\boldsymbol n) = 0 \bigg] \\
      &=  \E\bigg[ \E\big[|\det \nabla^2f_N(\boldsymbol n) |
          \indi\{i(\nabla^2f_N(\boldsymbol\sigma))=k, f(\boldsymbol n) \in \sqrt{N}B\} \big| f(\boldsymbol n) \big] \bigg].
    \end{split}
  \end{equation}
  By Lemma~\ref{l:conditioning}, the interior expectation satisfies
  \begin{equation}
    \label{e:bac}
    \begin{split}
      \E &\big[|\det \nabla^2f_N(\boldsymbol n) |
        \indi\{i(\nabla^2f_N(\boldsymbol\sigma))=k, 
          f(\boldsymbol n) \in \sqrt{N}B\} \big| 
        f(\boldsymbol n) \big] \\
      &= (2(N-1)p(p-1))^{\frac{N-1}{2}} 
      \E_\GOE^{N-1} \Big[\big|\det\big( M^{N-1}
          - p^{1/2}(2(N-1)(p-1))^{-1/2}f(\boldsymbol n)I\big)\big| \\
        &\quad\times \indi\Big\{i(M^{N-1}-p^{1/2}(2(N-1)(p-1))^{-1/2}f(\boldsymbol n)I)=k, 
          f(\boldsymbol n) \in \sqrt{N}B\Big\} \Big].
    \end{split}
  \end{equation}
  Inserting \eqref{e:bac} into \eqref{e:bab}, we can apply Lemma \ref{LemmaGOEAp} 
  with $m=0$, $t^2 = \frac{p}{2(N-1)(p-1)}$, 
  and $G = \sqrt{\frac{Np}{2(N-1)(p-1)}}B$. Using \eqref{e:aab} and
  \eqref{e:aac},
  we get after a little straightforward algebra,
  \begin{equation}
    \label{e:alternate}
    \E [\Crt_{N,k}(B)]
    = 2\sqrt\frac 2p (p-1)^{\frac N2} \E^N_{\GOE} \bigg[ e^{-N
        \frac{p-2}{2p} (\lambda^N_{k})^2} 
      \indi\Big\{ \lambda^N_{k} \in \sqrt{\frac{p}{2(p-1)}}B \Big\}
      \bigg].
  \end{equation}
  This completes the proof of Theorem \ref{t:exactk}. 
\end{proof}

\begin{proof}[Proof of Theorem~\ref{t:exactglobal}]
  Theorem~\ref{t:exactglobal} follows from Theorem~\ref{t:exactk} by
  summing over $k\in\{0,\dots,N-1\}$. The additional $N$ in the prefactor
  comes from the fact that $\rho_N$ is \textit{normalized} one-point
  correlation function, cf.~\eqref{e:rhonomr}.
\end{proof}
%>>>

\section{Logarithmic asymptotics of the complexity}%<<<
\label{s:ldp}

In this section we apply the LDP for the $k$-th largest eigenvalue of GOE, 
Theorem~\ref{t:LDP}, to study the logarithmic asymptotics of the complexity of 
the spherical spin glass, that is we show  Theorems \ref{t:complexityk} and 
\ref{t:complexityglobal}.  Observe, that comparing \eqref{e:Ms} with
\eqref{e:Ms2}, we must use Theorem~\ref{t:LDP} with $\sigma =2^{-1/2}$.
Here and later we write $\lambda_i$ for $\lambda^N_i$.

\subsection{Proof of Theorem \ref{t:complexityk}.}

We analyse the exact formula given in Theorem~\ref{t:exactk}, or equivalently 
in \eqref{e:alternate}. By Theorem \ref{t:LDP} and the obvious
symmetry between the largest and the smallest eigenvalues, the $(k+1)$-th smallest
eigenvalue $\lambda^N_{k}$ of $M^N$  satisfies
the LDP with the good rate function 
$J_k(u)=(k+1)I_1(-u;2^{-1/2})$, where $I_1$ is defined in \eqref{e:LDPrate}. 
Set $t=u \sqrt{\frac{p}{2(p-1)}}$ and 
$\phi (x)=-\frac{p-2}{2p}x^2$. Then, by Theorem~\ref{t:exactk},
\begin{equation}
  \label{e:cra}
  \lim_{N\to\infty} \frac 1N \log 
  \mathbb E \Crt_{N,k}(u)=
  \frac 12 \log (p-1)+
  \lim_{N\to\infty} \frac 1N \log 
  \mathbb E_{\GOE}^N\big[e^{N\phi (\lambda_{k}^2)}
    \indi_{\lambda_{k}\le t}\big].
\end{equation}
By Varadhan's Lemma (see e.g.~\cite{DZ98}, Theorem 4.3.1 and  Exercise~4.3.11, 
  observe that $\phi $ is bounded from above, so that condition (4.3.2) of 
  \cite{DZ98} is obviously satisfied)  
\begin{equation}
  \begin{split}
    \label{e:crb}
    \sup_{x\in (-\infty,t)}&(\phi (x)-J_k(x))
    \le 
    \liminf_{N\to\infty} \frac 1N \log 
    \mathbb E_{\GOE}^N\big[e^{N\phi (\lambda_{k}^2)}
      \indi_{\lambda_{k}< t}\big]
    \\&\le
    \limsup_{N\to\infty} \frac 1N \log 
    \mathbb E_{\GOE}^N\big[e^{N\phi (\lambda_{k}^2)}
      \indi_{\lambda_{k}\le t}\big]\le
    \sup_{x\in (-\infty,t]}(\phi (x)-J_k(x)). 
  \end{split}
\end{equation}
It can be seen easily that for 
$t\le -\sqrt 2 $ the both suprema in \eqref{e:crb} equal
$\phi (t)-J_k(t)$. On the other hand, if 
$t>-\sqrt 2$, these suprema equal $\phi (\sqrt 2)$. Hence, using  
the definitions of $\phi $, $t$, $J_k$, $E_{\infty}$, and the scaling relation \eqref{e:Iscal}, 
\begin{equation}
  \eqref{e:cra}=
  \begin{cases}
    \frac 12 \log (p-1)-\frac{p-2}{p},
    &\qquad\text{if $u> -E_{\infty}$,}\\
    \frac 12 \log (p-1)-\frac{(p-2)u^2}{4(p-1)}-
    (k+1)I_1(-u;{E_{\infty}}/2),
    &\qquad\text{if $u\le -E_{\infty}$.}\\
  \end{cases}
\end{equation}
Using Remark~\ref{r:Is}, this  completes the proof of Theorem~\ref{t:complexityk}.
\begin{remark}\label{r:remextension}
  The proof of Theorem~\ref{t:complexityk} clearly extends to a Borel set $B$. 
  In fact, one just need to apply Varadhan's Lemma and find the supremum of 
  $\phi (x)-J_k(x)$ on a appropriate domain.
\end{remark}

\subsection{Proof of Theorem \ref{t:complexityglobal}}

Let $t$ and $\phi $ as in the previous proof. By Theorem \ref{t:exactglobal}, we have 
to study 
\begin{equation}
  \label{e:balu1}
  \lim_{N\rightarrow \infty} \frac{1}{N} \log 
  2N\sqrt\frac 2p (p-1)^{\frac N2}
  \mathbb E \int_{-\infty}^t e^{N\phi (x)} L_N(\d x).
\end{equation}
where $L_N$ is the empirical spectral measure of the GOE matrix (see below 
  \eqref{e:Ms}). The constant in front the integral gives the term 
$\frac{1}{2} \log(p-1)$ of $\Theta_p(u)$. We need to evaluate the contribution 
of the integral. For $t\le - \sqrt 2$, using 
$\indi_{\lambda_0\le t}\ge \indi_{\lambda_i\le t}$ for all $i$, we write
\begin{equation}
  \frac 1N \mathbb E\big[e^{N\phi (\lambda_0)}\indi_{\lambda_0\le t}\big]\le
  \mathbb E \int_{-\infty}^t e^{N\phi (x)} L_N(\d x)
  \le \mathbb E\big[e^{N\phi (t)}\indi_{\lambda_0\le t}\big].
\end{equation}
Taking the logarithm and dividing by $N$, the both sides of this inequality
converge to $\phi (t)-I_1(-t;2^{-1/2})$, by the same argument as in the
previous proof. Using the values of $t$ and $\phi $, this proves the
theorem for $u\le -E_{\infty}$.

For $t\in (-\sqrt 2, 0]$, we cannot use the smallest eigenvalue $\lambda_0$ in the lower
bound. Therefore we write, for $\varepsilon >0$,
\begin{equation}
  \frac 1N \mathbb E\big[e^{N\phi (t-\varepsilon)}
    \indi_{L_N((t-\varepsilon ,t))>0}\big]
  \le
  \mathbb E \int_{-\infty}^t e^{N\phi (x)} L_N(\d x)
  \le \mathbb E\big[e^{N\phi (t)}\indi_{\lambda_0\le t}\big].
\end{equation}
By the LDP for $\lambda_0$, $\mathbb P[\lambda_0\ge t]\to 1$ as
$N\to\infty$. Similarly, by the convergence of $L_N$ to the semi-circle
distribution, we have $\mathbb P[ L_N((t-\varepsilon ,t))>0]\to 1$ as
$N\to\infty$. Therefore, after taking the logarithm and dividing by $N$, we
find
\begin{equation}
  \phi (t-\varepsilon)\le \lim_{N\to\infty} \frac 1N \log 
  \mathbb E \int_{-\infty}^{t} e^{N\phi (x)} L_N(\d x)
  \le \phi (t).
\end{equation}
Since $\phi $ is continuous, the claim of the theorem follows for
$t\in(-\sqrt 2,0]$, or equivalently for $u\in (-E_{\infty},0]$. 
The proof in the case $u>0$ is analogous and we left it to the reader.

%>>>

\section{The geometry of the bottom of the energy landscape} %<<<
\label{s:gs}

In this section, we use our complexity estimates to obtain information about the bottom of the energy landscape.
We first prove
Theorems~\ref{t:nofiniteindex} 
and \ref{t:belowEk}.
We then prove Theorem \ref{t:groundstate}, that is we show that
the normalized energy of the ground state  
\begin{equation}
  GS^N = N^{-1} \inf\{ H_{N,p}(\boldsymbol \sigma):\boldsymbol \sigma \in
    S^{N-1}(\sqrt N)\}
\end{equation}
converges to $-E_0(p)$ in probability.

\subsection{Proof of Theorem \ref{t:nofiniteindex}}

We want to show that there are no critical points of a finite index of the 
Hamiltonian above the level $N(-E_{\infty}+\varepsilon)$. Let  $k$ and
$\varepsilon $ be as in the statement of the theorem.
Then, by Markov's inequality and Theorem~\ref{t:exactk},
\begin{equation}
  \label{goolkk}
  \begin{split}
    \Pro & \big[\Crt_{N,k}((-E_{\infty}+\varepsilon,\infty ))>0\big] 
    \leq 
    \E \big[\Crt_{N,k}((-E_{\infty} + \varepsilon,\infty)) \big]
    \\ &\leq c(p) (p-1)^{N/2} \E 
    \bigg[\exp\Big\{-\frac{N(p-2)(\lambda_{k})^2}{2p} \Big\}
      \indi\{\lambda_{k}  \geq -\sqrt{2} + \varepsilon' \}\bigg] 
    \\ &\leq c(p) (p-1)^{N/2} \Pro \Big[\lambda_{k}  \geq -\sqrt{2} +
      \varepsilon'\Big],
  \end{split}
\end{equation}
where $\varepsilon '=\varepsilon \sqrt 2 /E_{\infty}$.
By the LDP for the empirical spectral measure $L_N$
(see Theorem 1.1 of \cite{BG97}), we know that for some
$C(\varepsilon')>0$
\begin{equation}
  \Pro \Big[\lambda_{k}  \geq -\sqrt{2} + \varepsilon\Big]
  \leq e^{-C_{\varepsilon}N^2}.
\end{equation}
Combining this estimate with  \eqref{goolkk} completes the proof of Theorem 
\ref{t:nofiniteindex}.

\subsection{Proof of Theorem \ref{t:belowEk}}

We want to prove that there are no critical values of index $k$ of the
Hamiltonian below
$-N(E_{k}+\varepsilon )$. Using Theorem~\ref{t:complexityk}, we have 
\begin{equation}
  \mathbb E\big[ \Crt_{N,k}(-E_{k}-\varepsilon )\big]\le 
  \exp\big\{N
    \Theta_{k,p}(-E_{k}-\varepsilon) +o(N)\big\}.
\end{equation}
The function $\Theta_{k,p}$ is strictly increasing on $(-\infty,-E_{\infty})$. The 
constant $E_{k}>E_{\infty}$ is defined by $\Theta_{k,p}(-E_{k})=0$. Therefore, 
$\Theta_{k,p}(-E_{k}-\varepsilon )= c(k,p,\varepsilon )<0$. An application of 
Markov's inequality as before completes the proof of Theorem~\ref{t:belowEk}.

\subsection{Proof of Theorem \ref{t:groundstate}}

Note that, by Theorem~\ref{t:belowEk}
there are no minima of the Hamiltonian
below $-N(E_0(p)+\varepsilon)$, with high probability. This implies that for
all $\varepsilon >0$
\begin{equation}
  \label{posareq}
  \lim_{N\to\infty}\mathbb P[GS^N\ge -E_0(p)-\varepsilon ]=1.
\end{equation}

To find a matching upper bound on $GS^N$ we use known results about the free 
energy at positive temperature, more precisely the Parisi formula as proved by 
Talagrand \cite{Tal06}. Recall that the partition function of the $p$-spin spin 
glass is given by
\begin{equation}
  Z_{N,p}(\beta)= \int_{S^{N-1}(\sqrt{N})} 
  e^{-\beta H_{N,p}(\boldsymbol \sigma)} 
  \Lambda_N(\d \boldsymbol \sigma),
\end{equation}  
where $\Lambda_N$ is the normalized surface measure on the sphere
$S^{N-1}(\sqrt N)$. By Theorem 1.1 of \cite{Tal06}, 
\begin{equation}
  \label{ParisiFor}
  \lim_{N\rightarrow \infty} \frac{1}{N} \E \log Z_{N,p}(\beta) = F_p(\beta),
\end{equation}
where $F_p(\beta)$ is given by the variational principle as in (1.11) of 
\cite{Tal06}. Furthermore, it is known that in the case of the spherical
$p$-spin model, the $F_p(\beta )$ can be computed using the following simpler
variational problem (see (1.16) in \cite{PT07}),
\begin{equation}
  F_p(\beta)=  
  \inf_{q,m \in [0,1]} \frac{1}{2} \Big\{ \beta^2\big(1 + (m-1)q^p\big) 
    + \Big(1-\frac{1}{m}\Big) \log(1-q) + \frac{1}{m} \log\big(1
      -q(1-m)\big) \Big\}.
\end{equation}

We now analyze this variational problem. We replace $q=1-d\beta^{-1}$ and 
$m=c\beta^{-1}$ in the last equation, and we denote the function inside the 
infimum by $P(\beta,c,d)$,
\begin{equation}
  \begin{split}
    P(\beta,&c,d)=
    \frac{1}{2} \Big\{ \beta^2\big(1 + (c\beta^{-1}-1)(1-d\beta^{-1})^p\big) 
      \\ &+ \big(1-c^{-1}\beta\big) 
      \log\big(1-(1-d\beta^{-1})\big) 
      + c^{-1}\beta 
      \log\big[1 -(1-d\beta^{-1})(1-c\beta^{-1})\big] \Big\}.
  \end{split}
\end{equation}
The following lemma shows that this is the right scaling as $\beta$ tends to 
infinity.
\begin{lemma} 
  There exist constants $0< \varepsilon <M <\infty$, such that, for all $\beta$ 
  large enough,
  \begin{equation}
    F_p(\beta) = \inf_{c,d \in [0,\beta]} P(\beta,c,d) 
    = \inf_{c,d \in [\varepsilon ,M]} P(\beta,c,d).
  \end{equation}
\end{lemma}
\begin{proof}
  It follows directly from Lemma 3 in \cite{PT07}.
\end{proof}

As $\beta \to \infty$, the function $\beta^{-1}P(\beta ,c,d)$ converges
uniformly on the compact set $c,d\in [\varepsilon ,M]$. Therefore the
last lemma implies that 
\begin{equation}
  \label{apro}
  \lim_{\beta \to \infty}
  \beta^{-1}F_p(\beta) 
  =\inf_{c,d \in [\varepsilon ,M]} \frac{1}{2}\Big\{  c 
    + p d + \frac{1}{c} \big(\log (c+d) -\log d \big) \Big\}
  =: \inf_{c,d \in [\varepsilon ,M]} P(c,d).
\end{equation}

\begin{lemma}
  \label{gammaequalE} 
  We have that 
  \begin{equation}
    \gamma \equiv \inf_{c,d \in [\varepsilon ,M]} P(c,d) = E_0(p). 
  \end{equation}
\end{lemma}

\begin{proof}
  The constant $E_0(p)$ is the unique solution of the equation
  $\Theta_{0,p}(-x)=0$ (see \eqref{e:thetakp}, \eqref{e:E_k}). Therefore, 
  to prove the lemma it suffices to show
  \begin{equation}
    \Theta_{0,p} (-\gamma) = 0.
  \end{equation}

  A critical point of $P(c,d)$ satisfies
  \begin{gather}
    \label{e:cda}
    p^{-1}=d(c+d)\\
    \label{e:cdb}
    c^2(c+d)+c = (c+d)(\log(c+d) - \log(d)).
  \end{gather}
  Writing $y=c+d$, it follows from \eqref{e:cda} that
  $d=(py)^{-1}$, $c=y-d=(py^2-1)/py$. From \eqref{e:cda}  
  and $d\ge \varepsilon $, we further obtain $d=\frac{1}{2}(-c + \sqrt{c^2+4p^{-1}})$.
  Therefore $c,d\ge \varepsilon $ implies
  \begin{equation}
    y = c+d > p^{-1/2}.
  \end{equation}
  Inserting these computations into \eqref{e:cdb}, we obtain that 
  $y$ is a solution of 
  \begin{equation}
    \label{e:eqy}
    \Big(\frac{py^2-1}{py}\Big)^2y + \frac{py^2-1}{py} = y\log(py^2),
    \qquad y>p^{-1/2}.  
  \end{equation}
  Setting $a=p y^2$, this is equivalent to 
  \begin{equation}
    \label{e:eqa}
    g_p(a):=(a-1)^2+p(a-1)-pa\log a=0,\qquad a>1.
  \end{equation}
  The function $g$ satisfies $g_p(1)=g'_p(1)=0$, $g''_p(1)<0$ and $g''$ is 
  increasing on $[1,\infty)$. Therefore, $g_p$ has a unique minimum $a_0$ on 
  $[1,\infty)$ and is strictly increasing on $[a_0,\infty)$. Moreover, 
  since $g_p(p-1)<0$, the equation \eqref{e:eqy} has a unique solution satisfying 
  \begin{equation}
    \label{e:ycond}
    a=py^2>p-1 \qquad\text{and thus}\qquad y\ge \sqrt{(p-1)/p}.
  \end{equation}

  Using the definition of $\gamma $ and \eqref{e:eqy}
  \begin{equation}
    \label{e:gamma}
    \gamma = \frac{1}{2}\Big(\frac{py^2-1}{py} + \frac{1}{y} +
      \frac{py}{py^2-1}\log(py^2)\Big)
    =y + \frac{p-1}{yp}.
  \end{equation}
  To compute $\Theta_{0,p}(-\gamma )$ observe that, by \eqref{e:gamma},
  $\gamma^2 - E_{\infty}^2 = \big(y-\frac{p-1}{py}\big)^2$. Hence, \eqref{e:ycond}
  implies
  \begin{equation}
    \gamma + \sqrt{\gamma^2 - E_{\infty}^2} = 2y.
  \end{equation}
  Inserting these results into definition \eqref{e:thetakp} of
  $\Theta_{0,p}$ and using equation \eqref{e:eqy} yields after a little
  algebra that $\Theta_{0,p}(-\gamma)=0$. This completes the proof of
  Lemma~\ref{gammaequalE}
\end{proof}

We can now prove the upper bound on $GS^N$. Note that
\begin{equation}
  \frac{1}{\beta N} \E \log Z_N =  
  \frac{1}{\beta N} \E
  \log \int_{S^N(\sqrt{N})} e^{-\beta H_{N,p}(\boldsymbol \sigma)}
  \Lambda_N(\d \boldsymbol \sigma ) \leq 
  - \E  GS^N.
\end{equation}
Taking the limits $N\to\infty$ and then $\beta \to \infty$, using \eqref{apro}
and Lemma \ref{gammaequalE}, we obtain
\begin{equation}
  \label{skasa32}
  \E GS^N \leq -E_0(p)+ \varepsilon \qquad \text{for $N$ large enough}.
\end{equation}
By Borell-TIS inequality (see Theorem 2.7 in \cite{AW09}),
\begin{equation}
  \label{dordec}
  \Pro \big[\big|GS^N + \E GS^N\big| > \varepsilon \big] \leq e^{-N\varepsilon^2}.
\end{equation}
Combining \eqref{skasa32} and \eqref{dordec}, we get that for all 
$\varepsilon > 0$, for $N$ large enough,
\begin{equation}
  \Pro \big[GS^N  \le -E_0(p) + 2\varepsilon \big] \ge 1-\varepsilon. 
\end{equation}
This combined with the lower bound \eqref{posareq} completes the proof of 
Theorem \ref{t:groundstate}.
%>>>

\section{The Thouless-Anderson-Palmer complexity}%<<<
\label{s:TAP}
In this section, we study the mean number of solutions of the
Thouless-Anderson-Palmer (TAP) equation \eqref{tapeq}, called the TAP complexity.
Using the previous results of this paper, we give, in Theorem~\ref{t:TAP}, a formula 
for the TAP complexity at any finite temperature. In physics literature
\cite{CS95}, the TAP complexity was predicted and used to derive a formula for
the complexity of the spherical $p$-spin model. In Lemma 6.3, we show that
the formula of \cite{CS95} agrees with our Theorem~\ref{t:complexityk}. For
further physics interpretation of the
TAP solutions, such as connections to metastable states, the reader is
invited to check \cite{KPV93,CS95} and the references therein.

Let $B(0,\sqrt{N})\subset \mathbb R^N$ be the open ball of radius $\sqrt{N}$ 
centered at $0$. We define the TAP 
functional as Kurchan, Parisi and Virasoro in \cite{KPV93}. For 
$\boldsymbol m = (m_1,\ldots,m_{N})$ in $B(0,\sqrt{N})$ and 
$q= \frac{1}{N} \sum_i m_i^2\in [0,1)$, let
\begin{equation}
  \label{tapfunc} 
  F_\TAP(\boldsymbol m) 
  = \frac{1}{2^{1/2}N^{(p+1)/2}} 
  \sum_{i_1, \ldots, i_p=1}^N 
  J_{i_1, \ldots i_p}m_{i_1}\cdots m_{i_p} + B_{p,\beta }(q),
\end{equation}
where, as usual, $J_{i_1, \ldots i_p}$ are independent standard normal random 
variables and
\begin{equation}
  \label{e:Bq}
  B(q)=B_{p,\beta }(q) = - \frac{1}{2\beta} \log (1-q) - \frac{\beta}{4} 
    \big( 1 + (p-1)q^p - pq^{p-1} \big).
\end{equation}

Notice that the TAP functional $F_\TAP$ can be written in spherical
coordinates: Defining
$\boldsymbol\sigma = q^{-1/2}\boldsymbol m\in S^{N-1}(\sqrt{N})$ and
$h_{N,p}(\boldsymbol \sigma )=N^{-1}H_{N,p}(\boldsymbol \sigma )$,
\begin{equation}
  F_\TAP(\boldsymbol m) 
  = f_\TAP(q,\boldsymbol\sigma) 
  =  2^{-1/2}q^{p/2} h_{N,p}(\boldsymbol\sigma) +B_{p,\beta }(q).
\end{equation}

The TAP equations are equations for critical points of the TAP functional,
\begin{equation}
  \label{tapeq}
  \frac{\partial}{\partial m_i} F_\TAP(\boldsymbol m) = 0.
\end{equation}
Since $q=0$ is not a critical point of $F_\TAP$, the equations
\eqref{tapeq} are equivalent with
\begin{gather}
  \label{e:condh}
  \frac{\partial}{\partial \sigma_i} 
  f_\TAP(q,\boldsymbol\sigma)=0 
  \quad \Longleftrightarrow \quad
    \frac{\partial}{\partial \sigma_i} H_{N,p}(\boldsymbol\sigma) = 0,\\
  \label{e:condq} 
  \frac{\partial}{\partial q} f_\TAP(q,\boldsymbol\sigma) 
  = 2^{-3/2} p q^{p/2-1} h_{N,p}(\boldsymbol\sigma) + B'(q) = 0.
\end{gather}
Therefore, a solution of the TAP equation \eqref{tapeq} must be a critical point of the 
Hamiltonian $H_{N,p}$ such that $q$ satisfies \eqref{e:condq}. A critical 
point $(q,\boldsymbol\sigma)$ of the TAP functional of index $k$ is thus either a 
critical point of $H_{N,p}$ of index $k$ satisfying \eqref{e:condq} and 
$\frac{\partial^2}{\partial q^2} f_\TAP(q,\boldsymbol\sigma) > 0$, or a 
critical point of $H_{N,p}$ of index $k-1$ satisfying \eqref{e:condq} and
$\frac{\partial^2}{\partial q^2} f_\TAP(q,\boldsymbol\sigma) < 0$.

Let $\mathcal{N}_k(u,\beta)$ represent the number of critical points of index 
$k$ of the TAP functional with normalized energy $h_{N,p}$ smaller than $u$ at temperature $\beta^{-1}$, 
\begin{equation}
  \mathcal{N}_k(u,\beta) =
  \sum_{(q,\boldsymbol\sigma):\nabla f_\TAP(q,\boldsymbol\sigma) = 0} 
  \indi \big\{ h_{N,p}(\boldsymbol\sigma) \in (-\infty,u], i(\nabla^2
      f_\TAP (q,\boldsymbol\sigma))=0\big\}.
\end{equation} 
For each $u \in (-\infty,-E_{\infty}]$, $p\geq 3$, and $\beta>0$, we further define
\begin{align}
  \label{e:betau}
  &\beta(u) = \frac{p}{2^{3/2}(p-1)}
  \Big(\frac{p}{p-2}\Big)^{\frac{p-2}{2}}
  \big(-u-\sqrt{u^2-E_{\infty}^2}\big)\\
  \label{eq:ustar}
  &u^{\star}(\beta) = \sup\big\{v\in(-\infty,-E_{\infty}]: \beta(v) < \beta \big\}.
\end{align}
Observe that $\beta (\cdot)$ is an increasing function with
$\lim_{u\to-\infty} \beta (u)=0$.

\begin{theorem}
  \label{t:TAP} 
  For all $\beta > 0$ and $p\geq 3$,
  \begin{itemize}
    \item[(a)] If $u \leq -E_{\infty}$, $k=0$
    \begin{equation}
      \lim_{N\rightarrow \infty} \frac{1}{N} \log \E
      \mathcal{N}_0(u,\beta) = \Theta_{0,p}(u^{\star}(\beta)\wedge u).
    \end{equation}

    \item[(b)] If $u \leq -E_{\infty}$, $k>0$,
    \begin{equation}
      \lim_{N\rightarrow \infty} \frac{1}{N} \log \E
      \mathcal{N}_k(u,\beta) =
      \Theta_{k-1,p}(u^{\star}(\beta)\wedge u).
    \end{equation}

    \item[(c)] If $\beta < \beta(-E_{(k-1)\vee 0}(p))$ or if
    $u < -E_{(k-1)\vee 0}(p)$, then  $\mathcal{N}_k(u,\beta)$
    tends to zero in probability as $N\to\infty$.
  \end{itemize}
\end{theorem}

\begin{proof}
  We start by proving parts (a) and (b). Computing $B'(q)$ and multiplying 
  the both sides of \eqref{e:condq} by $2^{3/2}\beta (1-q)/p$,
  \eqref{e:condq} is equivalent to
  \begin{equation}
    \label{e:basequad}
    2^{-1/2}(p-1) \beta^2\big((1-q)q^{p/2-1}\big)^2
    + h_{N,p}(\boldsymbol\sigma)\beta (1-q)q^{p/2-1}
    + \frac{2^{1/2}}{p}  = 0.
  \end{equation}
  Setting
  \begin{equation}
    \label{base}
    z=z(q)=(1-q)q^{p/2-1}
  \end{equation} 
  we obtain
  \begin{equation}
    \label{base23}
    \beta z = \frac{1}{2^{1/2}(p-1)}\Big(-h_{N,p}(\boldsymbol\sigma) 
      \pm \sqrt{h_{N,p}^2(\boldsymbol\sigma) - E_{\infty}^2}\Big).
  \end{equation}
  Thus, a solution  $(\boldsymbol\sigma,q)$ to the TAP equation
  \eqref{tapeq} is a critical 
  point $\boldsymbol \sigma $ of the Hamiltonian that satisfies \eqref{base} and \eqref{base23}. 
  
  The next lemma counts the number of solutions of equation \eqref{base}.
  \begin{lemma}
    \label{simplelemma}
    The function $ f(q)= (1-q)q^{p/2-1}-\lambda$, has exactly two, one 
    or no zeros in $q\in [0,1)$ if 
    $0 < \lambda < \frac{2}{p}(\frac{p-2}{p})^{\frac{p-2}{2}} $, 
    $\lambda = \frac{2}{p}(\frac{p-2}{p})^{\frac{p-2}{2}} $
    and $\lambda >\frac{2}{p}(\frac{p-2}{p})^{\frac{p-2}{2}}$, respectively.
  \end{lemma}

  \begin{proof}
    If $\lambda = 0$, then clearly $0$ and $1$ are the only zeros. Since 
      $f'(q)= q^{p/2-2}(p-{pq}-2))/2$,
    there is a unique critical point of $f$ in $(0,1)$, a maxima at 
    $q=\frac{p-2}{p}$. Now, varying $\lambda$ gives us the result of the
    lemma.
  \end{proof}

  Applying the lemma to equations \eqref{base}, \eqref{base23}, we see that 
  in order to $(\boldsymbol\sigma,q)$ be a critical point,
  $h_{N,p}(\boldsymbol\sigma)$ must satisfy (ignoring the equalities which
    have zero probability)
  \begin{equation}
    \label{eq:oiuek}
    0 < \frac{1}{2^{1/2}\beta(p-1)}
    \Big(-h_{N,p}(\boldsymbol\sigma) 
      \pm \sqrt{h_{N,p}^2(\boldsymbol\sigma) 
        - E_{\infty}^2}\Big)
    < \frac{2}{p}\Big(\frac{p-2}{p}\Big)^{\frac{p-2}{2}},
  \end{equation}
  which is equivalent to (cf.~\eqref{e:betau}, \eqref{eq:ustar})
  \begin{equation}
    \label{e:hcond}
    h_{N,p}(\boldsymbol \sigma )<u^\star(\beta ).
  \end{equation}
  Furthermore, if \eqref{e:hcond} is satisfied, the equations
  \eqref{base}, \eqref{base23} have two
  solutions $q_1$, $q_2$ such that  $0<q_1 < (p-2)/p <  q_2 <1$. 

  At a critical point, using \eqref{e:basequad} to compute $h_{N,p}$, and
  \eqref{base} to simplify the terms containing~$q^p$, 
  \begin{equation}
    \label{eq:oiuek2}
    \begin{split}
      \frac{\partial^2}{\partial q^2}& f_\TAP(q,\boldsymbol\sigma)
      \\
      &=\frac{p}{2}\Big(\frac{p}{2}-1\Big)q^{\frac{p}{2}-2}
      h_{N,p}(\boldsymbol\sigma) + \frac{1}{2\beta(1-q)^2}
      -\frac{\beta q^{p-3}p(p-1)}{4}((p-1)q  -p+2)
      \\&
      =\frac{p(p-1)}{8q\beta(1-q)^2}\Big(q-\frac{p-2}{p}\Big)
      \Big(\beta^2z^2-\frac{2}{p(p-1)}\Big).
    \end{split}
  \end{equation}
  Since, $z(q_1)=z(q_2)$, the above expression is positive for one of
  $q_1$, $q_2$ and negative for the remaining one. Hence, every critical
  point $\boldsymbol \sigma $ of index $k$ of the Hamiltonian
  with $h_{N,p}(\boldsymbol \sigma )$ as in \eqref{e:hcond} contributes
  one critical point of $f_\TAP$ with index $k$ and one of index $k+1$.
  Therefore, for every $k\ge 0$, $\beta >0$ and $u<-E_{\infty}$, $\mathbb P$-a.s., 
  \begin{equation}
    \frac{\mathcal N_{k}(u,\beta )}{ \Crt_{N,k-1}(u\wedge u^\star(\beta ))+
      \Crt_k(u\wedge u^\star (\beta ))}\in [1,2],
  \end{equation}
  where we defined $\Crt_{N,-1}(-\infty,c)=0$.
  Claims (a) and (b) then follows directly from Theorem~\ref{t:complexityk},
  using the observation $\Theta_{k,p}(u)\ge\Theta_{k+1,p}(u)$ for all $k\ge 0$ and
  $u\le-E_{\infty}$. 

  Claim (c) of the theorem then follows from (a), (b) using the Markov 
  inequality. 
\end{proof}

\subsection{Complexity of the minima: derivation of formula (11) of
  \texorpdfstring{\cite{CS95}}{Crisanti-Sommers}.}

In this short section, we verify that the results obtained in 
Theorem~\ref{t:complexityk} for minima, resp.~maxima, agree with the
formula proposed by A.~Crisanti and H.-J.~Sommers in \cite{CS95}.

\begin{lemma}
  Let $z=\frac{1}{p-1}(-u-\sqrt{u^2-\frac{2(p-1)}{p}})$ and 
  $u\leq-\sqrt{\frac{2(p-1)}{p}}$, then
  \begin{equation}
    \label{e:csformula}
    \Theta_{0,p}(\sqrt{2}u)
    =\frac{1}{2}\Big(\frac{2-p}{p} 
      - \log (\frac{p z^2}{2}) +\frac{p-1}{2}z^2 
      - \frac{2}{p^2z^2}\Big).
  \end{equation}
\end{lemma}
\begin{proof}
  First, note that for $u=-\sqrt{\frac{2(p-1)}{p}}$, we have that, by 
  \eqref{e:thetakp}, 
  \begin{equation}
    \Theta_{0,p}(\sqrt{2}u)=\frac{1}{2}\log(p-1)-\frac{p-2}{p},
  \end{equation}
  which agrees with the left-hand side of \eqref{e:csformula}, taking $z=-\sqrt{\frac{2}{p(p-1)}}$.

  Second, the derivative with respect to $u$ of the left-hand side of
  \eqref{e:csformula} is given by
  \begin{equation}
    \Theta_{0,p}'(\sqrt{2}u) =(p-1)^{-1}\big\{ 
      (2-p)u + 
      \sqrt{p(pu^2-2p+2)}\big\}.
    %(-1 + 1/(p-1)) u +
    %\sqrt{p/(p-1)} \sqrt{(p u^2)/(p-1)-2}
  \end{equation}
  On the other hand, the derivative of the right-hand side of
  \eqref{e:csformula} equals
  \begin{equation}
    \begin{split}
      \frac{1}{2} &\Big((4 (p-1)^2 (-1 - u/\sqrt{u^2-(2 (p-1))/p }))/(
          p^2 (-u - \sqrt{ u^2-(2 (p-1))/p })^3) \\&- (
          2 (-1 - u/\sqrt{u^2-(2 (p-1))/p }))/(-u -
          \sqrt{u^2-(2 (p-1))/p }) \\ &+ ((-1 - u/\sqrt{u^2-(2(p-1))/p}) (-u -
            \sqrt{u^2-(2(p-1))/p }))/(p-1)\Big).
    \end{split}
  \end{equation}
  After a lengthy, but straightforward, simplification, the last
  expression coincides with $\Theta_{0,p}'(\sqrt{2}u)$. This completes the
  proof.
\end{proof}
%>>>

\section{Sharper estimates of  the mean number of critical points} 
%<<<
\label{s:sharp}

In this section we prove the sharp asymptotic results on the complexity of 
spherical $p$-spin spin glass, that is Theorem~\ref{t:sharp}. To this end we  
analyse the exact formula \eqref{e:final}  proved in 
Theorem~\ref{t:exactglobal}. 
For $u> - E_{\infty}$, a
simple application of Laplace's
method is sufficient. For $u\le - E_{\infty}$, more care is needed.
We write the one-point correlation 
function $\rho_N(x)$ appearing in \eqref{e:final} as a function of the Hermite 
polynomials and to use the asymptotics of these polynomials given by the 
Plancherel-Rotach formula. 

\subsection{Sharp asymptotics in the bulk}
We start with the case $u>-E_{\infty}$ that is parts (c), (d) of
Theorem~\ref{t:sharp}.
Setting
\begin{equation}
  \label{e:CvF}
  C_{N,p}=2N\sqrt\frac 2p (p-1)^{N/2}, \qquad v = -u\sqrt\frac p{2(p-1)},
  \qquad \phi(x)=-\frac{(p-2)x^2}{2p},
\end{equation}
we rewrite the expectation of the global complexity \eqref{e:final} as
\begin{equation}
  \label{e:finalsimple}
  \mathbb E \Crt_N(u)=
  C_{N,p} \int_{-\infty}^{-v} e^{N\phi(x)}\rho_N(x)\,\d x
  =C_{N,p} \int_v^\infty e^{N\phi(x)}\rho_N(x)\,\d x.
\end{equation}
For the last equality we used the fact that $\phi$ and $\rho_N$ are even
functions.
The case $u>-E_{\infty}$ corresponds to $v<\sqrt 2$. Since $\rho_N$ converges
uniformly to
the density of the semi-circle law
$\rho (x)=\pi^{-1}\sqrt{2-x^2}\indi_{|x|<\sqrt 2}$, by Laplace's method,
the principal contribution to the integral \eqref{e:finalsimple} comes from
the boundary point $x=v$. Since $F'(v)=0$ iff $v=0$, we have
\begin{equation}
  \begin{split}
    \mathbb E \Crt_N(0)&=\frac 12 C_{N,p} e^{N \phi(0)} \rho (0) 
    \sqrt\frac{2\pi }{N F''_p(0)}(1+o(1)),\\
    \mathbb E \Crt_N(u)&=C_{N,p} e^{N \phi(v)} \rho (v)
    \big(N F'_p(v)\big)^{-1}(1+o(1)), \quad \text{for $u\in (-E_{\infty},0)$}. 
  \end{split}
\end{equation}
For $u>0$, $\mathbb E\Crt_N(u)=2 \mathbb E\Crt_N(0)(1+o(1))$, obviously.
Inserting back the definitions \eqref{e:CvF} completes the proof of 
Theorem~\ref{t:sharp}(c,d).

\subsection{Sharp asymptotics at and beyond the edge}
We first rewrite the one-point correlation function $\rho_N$ using the Hermite functions
$\phi_j$, $j\in \mathbb N$, given by
\begin{equation}
  \phi_j(x) = (2^jj!\sqrt{\pi})^{-1/2} H_j(x)
  e^{-\frac{x^2}{2}},
\end{equation} 
where $H_j$, $j\in \mathbb N$ are Hermite polynomials, $H_0(x) =1 $ and
$  H_j(x) = e^{x^2}(-\frac{\d}{\d x})^j e^{-x^2}$.
The Hermite functions are orthonormal functions in $\R$ with respect to Lebesgue 
measure. From \cite{Meh91}, pp.~$128$ and~$135$, it follows that 
\begin{equation}
  \label{e:onepointf}
  \rho_N(x)= N^{-1/2} \big( S_N(x) + \alpha_N(x) \big),
\end{equation}
 where
\begin{align}
  \label{e:SN}
  S_N(x) &= \sum_{i=0}^{N-1} \phi_i^2(\sqrt{N}x) +
  \Big(\frac{N}{2}\Big)^{1/2}\phi_{N-1}(\sqrt{N}x)
  \int_{-\infty}^{\infty}\varepsilon(\sqrt{N}x - t)\phi_{N}(t)\, \d t,\\
  \label{epsilonfunction}
  \varepsilon(x) &= \sign(x)/2,\\
  %\begin{cases}
  %  1/2, &\qquad \text{if }  x>0, \\
  %  0,   &\qquad \text{if }  x=0, \\
  %  -1/2,&\qquad \text{if }  x<0,
  %\end{cases}\\
  \label{alphadef}
  \alpha_N(x)&= \begin{cases}
    \phi_{2m}(\sqrt{N}x)\big\{\int_{-\infty}^{\infty} \phi_{2m}(t)dt\big\}^{-1},
    &\qquad \text{if } N= 2m +1, \\
    0,   &\qquad \text{if $N$ is even}.
  \end{cases}
\end{align}
The factor $N^{-1/2}$ in \eqref{e:onepointf} comes from a change of
variables and the fact that the  one-point correlation function $R_1$ in
\cite{Meh91} is not normalized to be a probability density.

Using the Chistoffel-Darboux formula  (\cite{Meh91}, p.~420),  the 
first term of \eqref{e:SN} satisfies 
\begin{equation}
  \label{CDarboux}
  \sum_{i=0}^{N-1} \phi_i^2(\sqrt N x) = N \phi_N^2(\sqrt N x)
  - \sqrt{N(N+1)}\phi_{N-1}(\sqrt N x)\phi_{N+1}(\sqrt N x),
\end{equation}
so that $\rho_N(x)$ depends on $\phi_{N-1}$, $\phi_N$ and $\phi_{N+1}$ only.
We now state the Plancherel-Rotach asymptotics of the Hermite functions in
domains of our interest as a 
lemma since it will be our main tool from now on. 
We use $\psi$, $\Psi $ and $h$ to denote the functions
\begin{equation}
  \label{e:psis}
  \begin{gathered}
    \psi(x)=|x^2-2|^{1/2}, \qquad
    I_1(v,\frac{1}{\sqrt{2}})(x) = \int_{\sqrt{2}}^{x} \psi (y) \d y,\\
    h(x) = \bigg|\frac{x - \sqrt{2}}{x + \sqrt{2}}\bigg|^{1/4} + 
    \bigg|\frac{x + \sqrt{2}}{x - \sqrt{2}}\bigg|^{1/4}.
  \end{gathered}
\end{equation}
\begin{lemma}
  \label{l:PR}
  There exists a $\delta_0 > 0$, such that for all $0< \delta < \delta_0$ the 
  following holds uniformly for $x$ in the given domains.
  \begin{itemize}
    \item[(a)]For $\sqrt{2} - \delta \ < x < \sqrt{2} + \delta $,
    \begin{equation}
      \label{PRA4}
      \begin{split}
        \phi_N(\sqrt{N}x) = &
        \frac{1}{(2N)^{1/4}} \bigg\{
          \bigg|\frac{x+\sqrt{2}}{x-\sqrt{2}}\bigg|^{1/4} 
          |f_N(x)|^{1/4} \Ai(f_N(x)) (1+O(N^{-1})) 
          \\&- \bigg|\frac{x-\sqrt{2}}{x+\sqrt{2}}\bigg|^{1/4} 
          \frac{1}{|f_N(x)|^{1/4}} \Ai'(f_N(x))(1+O(N^{-1})) \bigg\},
      \end{split}
    \end{equation}
    where $ f_N(x) = N^{2/3} \big\{ \frac{3}{2}I_1(v,\frac{1}{\sqrt{2}}) (x) \big\}^{2/3}$,
    and $\Ai(x)$ is the Airy function of first kind, 
    $\Ai(x)= \frac{2}{\pi} \int_{-\infty}^{\infty} 
    \cos \big(\frac{t^3}{3}+ tx\big) \d t.$

    \item[(b)] For $x > \sqrt{2} + \delta$,
    \begin{equation}
      \label{PRA5}
      \phi_N(\sqrt{N}x) =  
      \frac{e^{-NI_1(v,\frac{1}{\sqrt{2}}) (x)} h(x) }{\sqrt{4\pi\sqrt{2N}}}
      (1 + O(N^{-1})).
    \end{equation}
  \end{itemize}
\end{lemma}

\begin{proof}
  The lemma follows from \cite{DKMVZ99,DG07}. Formulas closest to 
  our formulation can be found in \cite{DG07}, pp.~20--22.  Under their 
  notation, our case has $c_N = \sqrt{2N}$, $h_N(x)=4$ and $m=1$. For 
  definitions of $f$ see (2.15), (2.16) of \cite{DKMVZ99}. For 
  the original Plancherel-Rotach asymptotics for the Hermite polynomials, the 
  reader can also check \cite{PR29} or \cite{Sze81}.
\end{proof}

For the rest of this section, we suppose that $\delta $ is small enough so
that Lemma~\ref{l:PR} holds.
We write $a_N\sim b_N$, if $a_N=b_N(1+o(1))$ as $N\to\infty$. For sequences 
$a_N(\delta )$, $b_N(\delta )$ which depend on $\delta $, we write 
$a_N\simd b_N$ if 
$\lim_{\delta \to 0}\lim_{N\to\infty}a_N(\delta )/b_N(\delta )=1$. Finally, if 
$a_N\le C(\delta ) b_N$ for some $C(\delta )<\infty$, we write 
$a_N=O_\delta (b_N)$.

We first analyse the integrals appearing in \eqref{e:SN} and
\eqref{alphadef}. We define 
\begin{equation}
  J_N(x):= 
  \int_{-\infty}^{\infty}\varepsilon(\sqrt{N}x - t) \phi_{N}(t)\, \d t=
  N^{1/2}\int_{-\infty}^{\infty}\varepsilon(x - s) \phi_{N}(s \sqrt N)\, \d s.
\end{equation}
\begin{lemma}
  \label{l:phiints}
  \begin{itemize}
    \item[(a)] As $N\to\infty$ 
    \begin{equation}
      \label{e:intphiN}
      \int_{-\infty}^{\infty} \phi_{N}(x)\, \d x =
      2 \int_{0}^{\infty} \phi_N(x) \,\d x 
      \sim 2 (2N)^{-1/4}.
    \end{equation}

    \item[(b)] If $x>\sqrt 2$ and $N\to\infty$ over odd integers, then
    $J_N(x)=O(N^c e^{-NI_1(v,\frac{1}{\sqrt{2}}) (x)})$. 

    \item[(c)] If $x>\sqrt 2$ and
    $N\to \infty$ over even integers, then
      $J_N(x)\sim (2N)^{-1/4}$.

    \item[(d)] Let $a_N\ge 0$ be such that $\lim_{N\to\infty}a_N N^{2/3}=0$. Then 
    \begin{equation}
      J_N(\sqrt 2 + a_N)\sim
      \begin{cases}
        \frac 2 3 (2N)^{-1/4},&\qquad\text{as $N\to \infty$ over even
          integers,}\\
        -\frac  1 3 (2N)^{-1/4},&\qquad\text{as $N\to \infty$ over odd
          integers.}
      \end{cases}
    \end{equation}
  \end{itemize}
\end{lemma}
\begin{proof}
  Claim (a) can be derived directly from Proposition~4.3 of 
  \cite{DG07}. 

  The Hermite functions are even (odd) for $N$ even (odd). Therefore,
  \begin{equation}
    J_N(x)=
    \begin{cases}
      \label{e:JNsym}
      N^{1/2}\sign x \int_0^{|x|} \phi_N(s\sqrt N)\,\d s,
      &\quad\text{$N$ even},\\
      -N^{1/2}\sign x \int_{|x|}^\infty \phi_N(s\sqrt N)\,\d s,
      &\quad\text{$N$ odd}.\\
    \end{cases}
  \end{equation}
  From Lemma~\ref{l:PR}(b) it follows that for $x>\sqrt 2$ and some $c,C<\infty$
  \begin{equation}
    \label{e:phiintsa}
    \int_x^\infty \phi_N(s\sqrt N)\,\d s
    \le C N^c \int_x^\infty e^{-NI_1(v,\frac{1}{\sqrt{2}}) (s)}\d s=O(N^c e^{-NI_1(v,\frac{1}{\sqrt{2}}) (x)}).
  \end{equation}
  To prove the last equality we used the fact that $I_1(v,\frac{1}{\sqrt{2}}) $ is strictly increasing on
  $(\sqrt 2, \infty)$ and Laplace's method. Claims (b), (c) are then direct
  consequences of \eqref{e:phiintsa} and claim~(a).

  In view of claim (a),  \eqref{e:JNsym} and  \eqref{e:phiintsa}, to prove (d) it suffices to show 
  \begin{equation}
    \label{e:phionethird}
    \int_{\sqrt 2 + a_N}^{\sqrt2 +\delta}  \phi_N(s\sqrt N)\d s\simd
    \frac 13 2^{-1/4}N^{-3/4}.
  \end{equation}
  To this end we use Lemma~\ref{l:PR}(a). 
  We first linearise the function $f_N$ appearing there.
  It can be proved easily from its definition that for all $N\ge 1$,
  uniformly over $x\in (\sqrt 2 -\delta , \sqrt 2 +\delta )$
  \begin{equation}
    \begin{split}
      \label{e:flin}
      &f_N(x)\simd 2^{1/2}N^{2/3}(x-\sqrt 2),\\
      &f_N'(x)\simd 2^{1/2} N^{2/3}.
    \end{split}
  \end{equation}
  Hence, after  substitution $f_N(x) = z$ we obtain using
  Lemma~\ref{l:PR}(a),
  \begin{equation}
   \label{e:flin2}
    \begin{split}
    \int_{\sqrt 2 + a_N}^{\sqrt2 +\delta}  \phi_N(s\sqrt N)\,\d s
      &\simd
      \frac {1}{(2N)^{1/4}} 
      \int_{f_N(\sqrt 2 + a_N) }^{f_N(\sqrt 2+\delta) } 
      \Big\{
      |2^{3/2}|^{1/4}(N^{2/3} 2^{1/2})^{1/4} \Ai( z )  
      \\&\quad-|2^{3/2}|^{-1/4}(N^{2/3} 2^{1/2})^{-1/4} \Ai'(z)\Big\}
    2^{-1/2}N^{-2/3}\, \d z.
    \end{split}
  \end{equation}
  The limits of the integral converge to $0$ and $\infty$ respectively.
  Using the well-known identity $\int_0^\infty \Ai(z)\,\d z=\frac 13$, the claim
  \eqref{e:phionethird} follows. This completes the proof of the lemma.
\end{proof}

We can now compute $\mathbb E \Crt_N(u)$ for $u\le - E_{\infty}$. By
\eqref{e:finalsimple}, \eqref{e:onepointf}, \eqref{CDarboux}, 
\begin{equation}
  \label{e:sumints}
  \mathbb E \Crt_N(u) = 
  C_{N,p} N^{-1/2} \sum_{i=1}^4
  \int_{v}^\infty e^{N\phi(x)}T_i(x)\, \d x 
  =: C_{N,p} N^{-1/2} \sum_{i=1}^4 I_i(v),
\end{equation}
where
\begin{align}
  \label{e:Tone}
  T_1(x) &= \alpha_N(x),
  \\\label{e:Ttwo}
  T_2(x) &= N \phi_N^2(\sqrt{N}x),
  \\\label{e:Tthree}
  T_3(x) &= - \sqrt{N(N+1)} \phi_{N-1}(\sqrt{N}x)\phi_{N+1}(\sqrt{N}x),
  \\\label{e:Tfour}
  T_4(x) &= \Big(\frac{N}{2}\Big)^{1/2} \phi_{N-1}(\sqrt{N}x) J_N(x).
\end{align}

\begin{proof}[Proof of Theorem~\ref{t:sharp}(a)]
  We first estimate the four integrals $I_1(v)$, \dots,
  $I_4(v)$ for $v>\sqrt 2$. 

  \begin{lemma}
    \label{l:Toneexp}
    For $N$ even $I_1(v)=0$. When $N\to\infty$ over odd integers, then
    \begin{equation}
      \label{e:Toneexp}
      I_1(v)  \sim  
      \frac {h(v)}{4 \pi^{1/2} N}
      \frac{e^{I_1(v,\frac{1}{\sqrt{2}}) (v)- \frac v2\psi (v)}}
      {-\phi'(v) + I_1(v,\frac{1}{\sqrt{2}})'(v)}\,
      e^{-N(F(v) + I_1(v,\frac{1}{\sqrt{2}})(v))} .
    \end{equation}
  \end{lemma}
  \begin{proof}
    The first claim follows directly from the definition \eqref{alphadef}
    of $\alpha_N$. To prove the second claim we fix $\delta < v-\sqrt 2$
    and use Lemmas~\ref{l:PR}(b), \ref{l:phiints}(a). Then,
    \begin{equation}
      \label{e:Toneexpa}
      I_1(v)\sim
      \frac{(2N)^{1/4}}{2}\int_{v}^\infty
      \frac{h(x_+)e^{N\phi(x)-(N-1)I_1(v,\frac{1}{\sqrt{2}}) (x_+)}}{(4\pi )^{1/2}(2(N-1))^{1/4}}
      \,\d x,
    \end{equation}
    where $x_+=x\sqrt{N/(N-1)}\sim x(1+1/(2N))$. Expanding
    $I_1(v,\frac{1}{\sqrt{2}}) (x_+)\sim I_1(v,\frac{1}{\sqrt{2}}) (x)+xI_1(v,\frac{1}{\sqrt{2}}) '(x)/(2N)$, we obtain the lemma
    by Laplace's method.
  \end{proof}

  The integrals $I_2(v)$, $I_3(v)$ are negligible for $v>\sqrt 2$:
  \begin{lemma}
    \label{l:Ttwoexp}
    There exists a $c<\infty$ such that 
    $I_i(v)=O(N^c e^{-N(-\phi(v)+2I_1(v,\frac{1}{\sqrt{2}}) (v))})$ for $i=2,3$.
  \end{lemma}
  \begin{proof}
    Observe that $T_2$, $T_3$ contain a product of two Hermite functions.
    By Lemma~\ref{l:PR}(b), these behave like $O(N^c e^{-NI_1(v,\frac{1}{\sqrt{2}}) (x)})$ for
    $x>\sqrt 2$, which implies the claim.
  \end{proof}

  \begin{lemma}
    \label{l:Tfourexp}
    \begin{itemize}
      \item[(a)] If $N\to \infty$ over odd integers, then
      $$I_4(v)=O(N^c e^{-N(-\phi(v)+2I_1(v,\frac{1}{\sqrt{2}}) (v))})$$ for some $c<\infty$.
      
      \item[(b)] If $N\to\infty$ over even integers, then $I_4(v)$ behaves like the
      right-hand side of \eqref{e:Toneexp}.
    \end{itemize}
  \end{lemma}
  \begin{proof}

    Claim (a) follows from Lemma~\ref{l:phiints}(b) and the same
    reasoning as in the previous proof. 
    For claim (b), we find using
    Lemma~\ref{l:phiints}(c) and Lemma~\ref{l:PR}(b) that $I_4$ is
    asymptotically equivalent to the right-hand side of \eqref{e:Toneexpa}
    and the claim follows.
  \end{proof}

  Theorem~\ref{t:sharp}(a) follows directly from the previous three lemmas
  and definitions \eqref{e:CvF}, \eqref{e:psis}. Observe that the dominant
  contribution comes from $I_1$ for $N$ odd and from $I_4$ for $N$ even.
\end{proof}

\begin{proof}[Proof of Theorem~\ref{t:sharp}(b)]
  We need now compute $I_1(v),\dots, I_4(v)$ for $v=\sqrt 2$.
  We split all these integrals into two parts: over
  $(\sqrt 2, \sqrt 2+\delta )$ and over $(\sqrt 2 +\delta ,\infty)$. For
  the second interval we can use Lemmas~\ref{l:Toneexp}--\ref{l:Tfourexp}. We
  need thus analyse the integrals over the first interval,
  $I_i^\delta=\int_{\sqrt 2}^{\sqrt 2+\delta }T_i(x)\,\d x$,
  $i=1,\dots, 4$.

  \begin{lemma}
    \label{l:Toneedge}
    For $N$ even $I_1(\sqrt 2)=0$. When $N\to\infty$ over odd integers, then
    \begin{equation}
      \label{e:Toneedge}
      I_1(\sqrt 2)\simd I_1^\delta \simd
      \frac {\Ai (0)}{2^{1/2} N^{5/6} F'_p(\sqrt 2)}\, e^{N\phi(\sqrt 2)}.
    \end{equation}
  \end{lemma}
  \begin{proof}
    Due to Lemma~\ref{l:Toneexp} it suffices to prove the second '$\simd$'
    relation only. By Lemma~\ref{l:PR}(a) and Lemma~\ref{l:phiints}(a),
    setting $x_+=x\sqrt{(N-1)/N}$,
    \begin{equation}
      \begin{split}
        I_1^\delta \simd
        \frac 12
        \int_{\sqrt 2}^{\sqrt 2+\delta }e^{N\phi(x)} 
        \bigg\{&\bigg|\frac{x_++\sqrt 2}{x_+-\sqrt 2}\bigg|^{1/4} 
          |f_N(x_+)|^{1/4} \Ai (f_N(x_+)) 
          \\&+
          \bigg|\frac{x_+-\sqrt 2}{x_++\sqrt 2}\bigg|^{1/4} 
          |f_N(x_+)|^{-1/4} \Ai' (f_N(x_+))
          \bigg\}\d x.
      \end{split}
    \end{equation}
    If it were not for the $f_N(x_+)$ in the argument of the Airy function, this
    integral could be analysed trivially by Laplace's method, the main
    contribution coming from  neighbourhoods of size $O(N^{-1})$ of
    $\sqrt 2$. However, due to \eqref{e:flin}, $f_N(x_+)$ converges to $0$
    uniformly over such neighbourhoods. Hence, by a simple extension of Laplace's
    method, using \eqref{e:flin} several times,
    \begin{equation}
      I_1^\delta \simd 
      \frac {\Ai (0)}{2^{1/2} N^{5/6} F'_p(\sqrt 2)}\, e^{N\phi(\sqrt 2)}.
    \end{equation}
    This completes the proof.
  \end{proof}

  \begin{lemma}
    \label{l:Ttwoedge}
    As $N\to\infty$, 
    \begin{equation}
      I_2(\sqrt 2)=O(N^{-1/6} e^{N\phi(\sqrt 2)}) 
      \qquad\text{and}\qquad
      I_3(\sqrt 2)=-I_2(\sqrt 2)(1+O(N^{-1})).
    \end{equation}
  \end{lemma}
  \begin{proof}
    As before, $I_2(\sqrt 2)\simd I_2^\delta $. To estimate $I_2^\delta $ we
    use Lemma~\ref{l:PR}(a) again. The same computation as in the proof of
    the previous lemma gives 
    \begin{equation}
      I_2^\delta \simd 
      \frac{\Ai^2(0) \sqrt 2 }{N^{1/6} F'_p(\sqrt 2)}\, e^{N\phi(\sqrt 2)}.
    \end{equation}
    To prove the second claim, it suffices to observe that on
    $(\sqrt 2, \sqrt 2 + \delta )$, by Lemma~\ref{l:PR}(a), 
    \begin{equation}
      \frac{\phi_N^2(x\sqrt N)}
      {\phi_{N-1}(x\sqrt N)\phi_{N+1}(x\sqrt N)}
      =1+O(N^{-1}).
    \end{equation}
  \end{proof}

  \begin{lemma} 
    \label{l:Tfouredge}
    \begin{itemize}
      \item[(a)]   When $N\to\infty$ over even integers, then
      \begin{equation}
        I_4(\sqrt 2)\simd I_4^\delta \simd \frac 23 G_N(p)
      \end{equation}
      where $G_N(p)$ denotes the right-hand side of \eqref{e:Toneedge}. 
      \item[(b)] When $N\to\infty$ over odd integers, then
      \begin{equation}
        I_4(\sqrt 2)\simd I_4^\delta \simd -\frac 13 G_N(p).
      \end{equation}
    \end{itemize}
  \end{lemma}
  \begin{proof}
    The lemma follows by the same computations in Lemma~\ref{l:Toneedge}. 
    Again, the dominant contribution comes from neighbourhoods of size
    $O(N^{-1})$ of $\sqrt 2$. On such neighbourhoods, using
    Lemma~\ref{l:phiints}(d), the integral
    $J_N(x)$ appearing  in $T_4(x)$ can be approximated by $cN^{-1/4}$,
    where $c$ depends on $N$ being even or odd. 
  \end{proof}

  Theorem~\ref{t:sharp}(b) follows now directly from the previous three lemmas 
  and definitions \eqref{e:CvF}, \eqref{e:psis}. Observe that the dominant 
  contribution comes from $I_1$ and $I_4$, since $I_2$ and $I_3$ mutually 
  cancel.  Moreover, the combined contributions of $I_1$ and $I_4$ do not 
  depend on $N$ being odd or even.  
\end{proof}

\begin{proof}[Proof of Corollary~\ref{c:sharp}] 
  By Theorem \ref{t:complexityk}, for all  $\varepsilon>0$,
  \begin{equation}
    \begin{split} 
      \E \Crt_{N}(u) &\geq \E \Crt_{N,0}(u) \geq  \E \Crt_{N}(u)- N \E \Crt_{N,1}(u) \\ 
      &\geq   \E \Crt_{N}(u)- N e^{N (\Theta_{1,p}(u)+\varepsilon)}.
    \end{split}
  \end{equation}
  The corollary then follows from the fact 
  that for $u<-E_{\infty}$, $\Theta_{0,p}(u)>\Theta_{1,p}(u)$.
\end{proof}
%>>>

\appendix
\section{Large deviations for the largest eigenvalues of the GOE} %<<<
\label{a:LDPsection}

In this appendix we extend Theorem~6.2 of \cite{BDG01}, proving a LDP 
for the $k$-th largest eigenvalue of the 
Gaussian Orthogonal Ensemble. 
This results might be of independent interest. 
Its proof follows the lines of \cite{BDG01}.

Let $X=X_N$ be a $N \times N$ real symmetric random matrix whose entries 
$X_{ij}$ are independent (up to the symmetry) centered Gaussian random 
variables with the variance
\begin{equation}
  \label{e:Ms2}
  \mathbb E X_{ij}^2 = \sigma^2N^{-1}(1+\delta_{ij}),
\end{equation}
and let $\lambda_1\le\dots\le\lambda_{N}$ be the ordered eigenvalues of
$X$. Note, that in this appendix we use the notation that is more usual in the
random matrix theory, that is the eigenvalues are numbered from $1$ to
$N$, not from $0$ to $N-1$ as in the rest of the paper. 
We show the following LDP.

\begin{theorem}
  \label{t:LDP}
  For each fixed $k\ge 1$, the $k$-th largest eigenvalue $\lambda_{N-k+1}$
  of $X$ satisfies a LDP with speed $N$ and 
  a good rate function 
  \begin{equation}
    I_k(x;\sigma) = k I_1(x;\sigma ) = 
    \begin{cases}
      \label{e:LDPrate}
      k\int_{2\sigma }^x\sigma^{-1} \sqrt{(\frac{z}{2\sigma })^{2}- 1}\, \d z,
      &\quad \text{if }  x \geq 2\sigma , \\
      \infty, &\quad \text{otherwise}.
    \end{cases}
  \end{equation}
\end{theorem}

\begin{proof}
  We first recall some know fact about the distribution of $\lambda_i$'s and
  introduce some notation.
  The joint law of $\lambda_i$, $1\leq i\leq N$, is given by
  \begin{equation}
    \label{e:LawGOE}
    Q_N (\d \lambda_1, \ldots, \d \lambda_N) = 
    Z_{N}(\sigma )^{-1}
    \prod_{1\leq i<j \leq N}
    |\lambda_i - \lambda_j|
    \prod_{i=1}^{N}\exp \Big(-\frac{N}{4\sigma^2} \lambda_i^2\Big) \d
    \lambda_i\indi_{\lambda_1\le \dots \le \lambda_N}.
  \end{equation}
  The distribution of unordered eigenvalues $\bar Q_N$ is given by
  the same formula without the final indicator function and with
  $Z_N(\sigma )$
  replaced by $\bar Z_N(\sigma )=N! Z_N(\sigma )$.
  By Wigner's theorem, the spectral measure
  $L_N=N^{-1}\sum_{i=1}^N\delta_{\lambda_i}$ of $X_N$ converges weakly in
  probability to the semi-circle distribution  
  \begin{equation}
    \label{e:rhosigma}
    \rho(\d x) = (2\pi\sigma^2)^{-1}\sqrt{(2\sigma)^2-x^2}\,\indi_{|x|\le
      2\sigma }\, \d x.
  \end{equation}
  For  $A \subset \R$ we denote by $\mathcal P(A)$ the space of all Borel 
  probability measures on $A$ endowed with the weak topology and a compatible 
  metric $d$.  By Theorem~1.1 of \cite{BG97}, the spectral measure $L_N$ 
  satisfies a LDP on $\mathcal P(\mathbb R)$ with the speed $N^2$ and a good 
  rate function  whose unique minimiser is the semi-circle 
  distribution~\eqref{e:rhosigma}.
  
  We now start proving Theorem~\ref{t:LDP}. $I_k(x;\sigma )$ is obviously
  good rate function. Since $I_k(x,\sigma)$ is scale invariant, that is 
  \begin{equation}
    \label{e:Iscal}
    I_k(x;\sigma )=I_k(x/\sigma ;1),
  \end{equation}
  we can assume that $\sigma=1$ as in \cite{BDG01}, and omit $\sigma $ from the 
  notation. Note that this value differs from the value  $\sigma =2^{-1/2}$ 
  used in the rest of the present paper (cf.\eqref{e:Ms}). In 
  particular, for $\sigma =1$ (cf.~\eqref{e:defZ})
  \begin{equation}
    \label{e:defZsigma}
    \bar Z_N=\bar Z_N(1)=(2\sqrt 2)^N \Big(\frac N2\Big)^{-N(N+1)/4}
    \prod_{i=1}^N \Gamma \Big(1+\frac i2\Big).
  \end{equation}

  To show Theorem~\ref{t:LDP} it is sufficient to prove
  \begin{equation}
    \label{e:LargeDeviation1}
    \limsup_{N\rightarrow \infty} \frac{1}{N} \log Q_N \big(\lambda_{N-k+1}
      \leq x\big) = -\infty \qquad \text{for all $x<2$},
  \end{equation}
  and, since $I_k(x;1 )$ is continuous and strictly increasing on
  $[2 , \infty)$,
  \begin{equation}
    \label{e:LargeDeviation2} 
    \limsup_{N\rightarrow \infty} \frac{1}{N} \log 
    Q_N \big(\lambda_{N-k+1} \geq x\big) = -I_k(x;1 )\qquad \text{for all
      $x\ge 2$}.
  \end{equation}

  We first prove \eqref{e:LargeDeviation1}. Suppose that
  $\lambda_{N-k+1}\le x$ for some $x<2 $. Then
  $L_N((x,2])\le(k-1)N^{-1}$. Since $\rho ((x,2])>0$, for
  $N$ large enough there exists a closed set
  $A\subset \mathcal P(\mathbb R)$ such that $\rho \notin A$ and
  $\{\lambda_{N-k+1}\le x\}\subset A$. By the LDP for the spectral measure $L_N$,
  $Q_N(A)\le e^{-c N^2}$ for some $c>0$, which concludes the proof of
  \eqref{e:LargeDeviation1}.

  We now prove the upper bound for \eqref{e:LargeDeviation2}. By Lemma~6.3 of 
  \cite{BDG01}, 
  \begin{equation}
    \label{e:fink}
    \bar Q_N\Big( \max_{i=1}^{N} |\lambda_i| \geq M\Big) \leq e^{-NM^2/9}
    \qquad
    \text{for all $M$ large enough and all $N$.}
  \end{equation}
  Writing
  \begin{equation}
    Q_N(\lambda_{N-k+1}\ge x)\le Q_N\big(\max_{i=1}^N|\lambda_i|\ge M\big)
    +Q_N\big(\lambda_{N-k+1}\ge x,\max_{i=1}^N|\lambda_i|<M\big),
  \end{equation}
  the upper bound follows easily provided we show that for all $M>x>2 $   
  \begin{equation}
    \label{e:slld}
    \limsup_{N \rightarrow \infty} \frac{1}{N} \log  
    Q_N\big( \max_{i=1}^N |\lambda_i| \leq M, \lambda_{N-k+1} \geq x\big) 
    \leq -
    I_k(x;1 ).
  \end{equation}

  To show \eqref{e:slld} we introduce some additional notation. Let
  $\bar Q_{N-k}^N$ be a measure on $\mathbb R^{N-k}$ given by
  \begin{equation}
    \label{e:Qredux}
    %\label{sigmakk}
    \bar Q_{N-k}^N\big(\lambda \in \cdot \big) = \bar Q_{N-k} \big( (1-kN^{-1})^{1/2} \lambda \in \cdot \big).
  \end{equation}
  We set 
  \begin{equation}
    \label{e:defCnk}
    C_N^k=\Big(1-\frac k N\Big)^{(N-k)(N-k+1)/4}
    \frac{\bar Z_{N-k}}{\bar Z_N},
  \end{equation}
  and for $x\in \mathbb R$ and $\mu \in \mathcal P(\mathbb R)$ we define
  \begin{equation}
    \Phi(z,\mu) = \int_{\mathbb R} \log |z-y| \mu(\d y) -
    \frac{z^2}{4}.
  \end{equation}
  It was shown in \cite[p.~50]{BDG01} that $\Phi(z,\mu)$ is upper 
  semi-continuous on $[-M,M]\times \mathcal P([-M,M])$ and continuous on 
  $[x,y] \times \mathcal P([-M,M])$ for all $M, x, y \in \R$ such that
  $y>x>M>2$.

  Using \eqref{e:LawGOE} and this notation, we can write
  \begin{equation}
    \begin{split}
      \label{e:LDa}
      Q_N\big(& \max_{i=1}^N |\lambda_i| \leq M, \lambda_{N-k+1} \geq x\big)\\
      &=
      Z_N^{-1} 
      \int_{[x,M]^{k}} \prod_{i=N-k+1}^Ne^{-N\lambda_i^2/4}\d \lambda_i
      \int_{[-M,M]^{N-k}} \prod_{i=1}^{N-k}e^{-N\lambda_i^2/4}\d \lambda_i\\
      &\qquad\times \prod_{1\leq i<j \leq N} |\lambda_i - \lambda_j| 
      \indi_{\lambda_1\le \dots \le \lambda_N}\\
      &\le C_N^k \frac {N!}{(N-k)!}
      \int_{[x,M]^k} \prod_{N-k< i<j \leq N} |\lambda_i - \lambda_j|
      \d \lambda_{N-k+1}\dots \d \lambda_N
      \\&\quad\times \int_{[-M,M]^{N-k}}
      e^{(N-k)\sum_{i=N-k+1}^N\Phi (\lambda_i,L_{N-k})}
      \bar Q_{N-k}^{N}(\d \lambda_1,\dots, \d \lambda_{N-k}),
      \end{split}
  \end{equation}
  where the factor $N!/(N-k)!$ comes from replacing $Q_N$ by $\bar Q_N$. Let 
  $B(\rho, \delta)$ denote the open ball in $\mathcal P(\mathbb R)$ of radius 
  $\delta >0$ and center $\rho$. We write 
  $B_M(\rho,\delta)=B(\rho,\delta )\cap \mathcal P([-M,M])$. On the domain of 
  the integration   $|\lambda_i-\lambda_j|\le 2M$ and 
  $e^{(N-k)\Phi (\lambda_i, L_{N-k})}\le (2M)^{N-k}$. Therefore by splitting 
  the second integral, \eqref{e:LDa} is bounded from above by
  \begin{equation}
    \begin{split}
      \label{e:LDb}
      C_N^k\frac {N!}{(N-k)!}
      (2M)^{k(k-1)/2}
      \Big\{&\Big(
        \int_{x}^M  e^{(N-k)\sup_{\mu \in B_M(\rho ,\delta )}\Phi(z ,\mu) }
        \,\d z\Big)^k
      \\&+(2M)^{N-k} \bar Q_{N-k}^N(L_{N-k}\notin B(\rho ,\delta ))
      \Big\}.
  \end{split}
  \end{equation}

  To control the measure $\bar Q_{N-k}^N$, observe that 
  for all functions $h:\mathbb R\to \mathbb R$ of Lipschitz norm at most $1$ and 
  $N \geq 2k$,
  \begin{equation}
    \Big|(N-k)^{-1} \sum_{i=1}^{N-k}\big\{ 
      h\big((1-kN^{-1})^{1/2} \lambda_i\big) -
      h(\lambda_i)\big\} \Big|
      \leq c N^{-1} \max_{i=k}^{N-k} |\lambda_i|
  \end{equation}
  for some $c\in (0,\infty)$ independent of $N$ and $k$. It follows from
  \eqref{e:fink} that the laws of $L_{N-k}$ under $\bar Q_{N-k}$ and
  $\bar Q^N_{N-k}$ are exponentially equivalent as $N\to \infty$.
  Therefore,  by Theorem~4.2.13 of \cite{DZ98}, $L_{N-k}$ under 
  $\bar Q^{N}_{N-k}$ satisfies the same LDP as $L_{N-k}$ under
  $\bar Q_{N-k}$.  Hence, the 
  second term in \eqref{e:LDb} is exponentially negligible for any 
  $\delta > 0$ and $M < \infty$. This implies that
  \begin{equation}
    \begin{split}
      \label{eq:psoapsd}
      \limsup_{N \rightarrow \infty} \frac{1}{N} \log  
      Q_N&\big( \max_{i=1}^{N} |\lambda_i| \leq M, \lambda_{N-k+1} \geq x\big) 
      \\&\leq \limsup_{N \rightarrow \infty} \frac{1}{N} \log C_N^k  
      + k \lim_{\delta \downarrow 0} \sup_{z \in [x,M], \mu \in B_M(\rho, \delta)} \Phi(z,\mu).
    \end{split}
  \end{equation}
  The same reasoning as on p.~50 of \cite{BDG01} implies that the second term in 
  \eqref{eq:psoapsd} equals $-k(1/2 + I_1(x;1))$.  From definition 
  \eqref{e:defCnk} of $C_N^k$ and from \eqref{e:defZsigma}, it is easy to 
  obtain $\lim_{N\rightarrow \infty} N^{-1} \log C_N^k = k/2$. Combining these 
  two claims,  we can bound the left-hand side of \eqref{e:slld} by 
  $-k(1/2+I_1(x;1))+k/2=k I_1(x;1)$. This completes the proof of \eqref{e:slld} 
  and thus of the upper bound for \eqref{e:LargeDeviation2}. 

  To prove the complementary lower bound we fix $ y > x > r > 2$ and
  $\delta > 0$. By a similar computation as in \eqref{e:LDa} we obtain
  \begin{equation}
    \begin{split}
      Q_N\big(&\lambda_{N-k+1} \geq x\big)  
      \geq \bar Q_N\big(\lambda_N \in [x,y], \dots, \lambda_{N-k+1} \in [x,y], 
        \max_{i=1}^{N-k} |\lambda_i| \leq r \big) 
      \\ &= C_N^k \int_{[x,y]^k} \prod_{i=N-k+1}^{N}e^{-k \lambda_i^2/4} \d
      \lambda_i
      \prod_{N-k < i < j \leq N} |\lambda_i - \lambda_j| 
      \\ &\quad\times \int_{[-r,r]^{N-k}}
      e^{(N-k)\sum_{i=N-k+1}^{N}\Phi(\lambda_i,L_{N-k})} 
      \bar Q_{N-k}^N (\d \lambda_1,\dots,\d \lambda_{N-k})
      \\&\ge  K C_N^k
      \exp \Big(k (N-k) \inf_{z \in [x,y], \mu \in B_r (\rho, \delta)} 
        \Phi (z, \mu)\Big) 
      \bar Q^N_{N-k}\big(L_{N-k} \in B_r(\rho, \delta)\big),
    \end{split}
  \end{equation}
  for some $K= K(k,x,y)>0$.

  Using the LDP for the measure $L_{N-k}$ under
  $\bar Q_{N-k}^N$ we see that
  \begin{equation}
    \bar Q_{N-k}^N\big(L_{N-k}
      \notin B_r(\sigma, \delta)\big) \xrightarrow{N\to\infty} 0.
  \end{equation} 
  By symmetry of $\bar Q_N(\cdot)$ and by the upper bound in
  \eqref{e:LargeDeviation2},
  \begin{equation}
    \bar Q^N_{N-k}\big(L_{N-k} \notin \mathcal P((-r, r))\big) 
    \leq 2 \bar Q_{N-k} (\lambda_{N-k} \geq r)\xrightarrow{N\to\infty} 0.
  \end{equation}
  Therefore, using the behavior of $C_N^k$ again,
  \begin{equation}
    \liminf_{N\rightarrow \infty} \frac{1}{N} \log Q_N
    \big(\lambda_{N-k+1} \geq x\big) \geq \frac{k}{2} + k \inf_{z \in [x,y], \mu
      \in B_r (\rho, \delta)} \Phi (z, \mu).
  \end{equation}
  Letting now $\delta \to 0$ and then $y\searrow x$, using the continuity
  of $\Phi (z, \mu)$ in the used range of the parameters, we obtain the
  desired lower bound.
\end{proof}
%>>>

\bibliography{complexity}
\bibliographystyle{amsalpha}
\end{document}